\documentclass[a4paper]{article}
\usepackage{amssymb}
\usepackage{graphicx}
\usepackage{amsmath}
\usepackage[T1]{fontenc}
\usepackage[usenames, dvipsnames]{color}
\usepackage{booktabs}
\definecolor{light-gray}{gray}{0.85}
\usepackage{graphicx}
\usepackage{enumitem}
\usepackage{caption}
\usepackage{subcaption}
\usepackage{changepage}
\usepackage[draft]{hyperref}

\usepackage[english]{babel}

\newtheorem{theorem}{Theorem}[section]

\newtheorem{corollary}{Corollary}[section]

\newtheorem{lemma}{Lemma}[section]

\newtheorem{proposition}{Proposition}[section]

\newenvironment{proof}[1][Proof]{\textbf{#1.} }{\ \rule{0.5em}{0.5em}}
\def\X{{\cal X}}
\def\Y{{\cal Y}}
\def\Z{{\cal Z}}
\def\B{{\cal B}}

\def\ZZ{{\bf Z}}

\def\pp{\psi_r}

\def\DEF{\stackrel{\mbox{\scriptsize{\normalfont{def}}}}{=}}
\makeatletter
\newsavebox\myboxA
\newsavebox\myboxB
\newlength\mylenA
\usepackage{titlesec}
\usepackage[titletoc,toc,title]{appendix}
\usepackage{titlesec}
\usepackage[titletoc,toc,title]{appendix}
\makeatletter
\def\namedlabel#1#2{\begingroup
   \def\@currentlabel{#2}%
   \label{#1}\endgroup
} \makeatother

\usepackage{authblk}
\begin{document}
\title{Exponential forgetting of  smoothing distributions for pairwise Markov models}
\date{}
\author{J\"uri Lember}
\author{Joonas Sova}
\affil{\small University of Tartu, Liivi 2 50409, Tartu, Estonia.\\
Email: \textit{jyril@ut.ee}; \textit{joonas.sova@gmail.com}} \maketitle

\begin{abstract}
We consider a bivariate Markov chain $Z=\{Z_k\}_{k \geq
1}=\{(X_k,Y_k)\}_{k \geq 1}$ taking values on product space $\Z=\X
\times \Y$, where $\X$ is possibly uncountable space and
$\Y=\{1,\ldots, |\Y|\}$ is a finite state-space.  The purpose of the
paper is to find sufficient conditions that guarantee the
exponential convergence of smoothing, filtering and predictive
probabilities: $$\sup_{n\geq t}\|P(Y_{t:\infty}\in
\cdot|X_{l:n})-P(Y_{t:\infty}\in \cdot|X_{s:n}) \|_{\rm TV} \leq K_s
\alpha^{t}, \quad \mbox{a.s.}$$ Here $t\geq s\geq l\geq 1$, $K_s$
is $\sigma(X_{s:\infty})$-measurable finite random variable and
$\alpha\in (0,1)$ is fixed. In the second part of the paper, we
establish two-sided versions of the above-mentioned convergence. We
show that the desired convergences hold under fairly general
conditions. A special case of above-mentioned very general model is
popular hidden Markov model (HMM). We prove that in HMM-case, our
assumptions are more general than all similar mixing-type of
conditions encountered in practice, yet relatively easy to verify.
\end{abstract}

\section{Introduction}
We consider a bivariate Markov chain $Z=\{Z_k\}_{k \geq
1}=\{(X_k,Y_k)\}_{k \geq 1}$ defined on a probability space $(\Omega, {\cal F}, {\bf P})$ and taking values on product space $\Z=\X
\times \Y$, where $\X$ is possibly uncountable space and
$\Y=\{1,\ldots, |\Y|\}$ is a finite set, typically referred to as the {\it state-space}. Process
$X=\{X_k\}_{k \geq 1}$ is seen as the observed sequence and
$Y=\{Y_k\}_{k \geq 1}$ is seen as hidden or latent variable
sequence, often referred to as  {\it the signal process}. The
process $Z$ is sometimes called the \textit{pairwise Markov model}
(PMM) \cite{pairwise, pairwise2, pairwise3, pairwise4} and covers
many latent variable models used in practice, such as hidden Markov
models (HMM) and autoregressive regime-switching models. For a
classification and general properties of pairwise models, we refer
to \cite{pairwise, pairwise3, pairwise4}. Generally, neither $Y$ nor
$X$ is a Markov chain, although for special cases they might be. In
many practical models, such as above-mentioned HMM's and Markov
switching models, the signal process $Y$ remains to be a Markov
chain. However, for every PMM, conditionally on the realization of
$X$ (resp. $Y$), the $Y$  (resp. $X$ process) is always an
inhomogenous Markov chain. The fact that we consider finite $\Y$
might seem restrictive at the first sight. The study of such models
is mainly motivated by the fact that in the most applications of
PMM's, specially of HMM's, the state space is finite, often rather
small and so it is clear that this case needs special treatment. Strictly speaking, the term "hidden Markov model" refers to the case of discrete ${\cal Y}$, the models with uncountable
state space ${\cal Y}$ are often called "state-space models" (see e.g. \cite{HMMbook}). Their difference is not only the level of mathematical abstraction, rather than different research objectives, techniques and algorithms  -- the finite $\Y$ allows effectively use many classical HMM tools like Viterbi, forward-backward and Baum-Welch algorithm, and  under finite  $\Y$ all these tools are applicable also for PMM case.
Thus the model considered in the present article could be considered
as a generalization of standard HMM, where the state space is still
finite, but the structure of the model is more involved allowing
stronger dependence between the observations. It turns out that with
finite $\Y$ many abstract
 conditions simplify so that they are easy to apply in
practice and many general conditions can be weakened. Also,
finite $\Y$ allows us to employ different technique. The observation
space $\X$, however, is very general, as it usually is in practice.
\\\\
In the current paper, the main object of interest  is the {\it
conditional signal process}, i.e. the process $Y$ conditioned on
$X$. More specifically, the purpose of the present work is to study
the distributions $P(Y_{t:t+m-1} \in \cdot|X_{s:n})$, where $m \geq
1$, $\infty\geq n \geq t \geq s$ and where we adopt the notation
$a_{l:n}$ for any vector $(a_l,\ldots,a_n)$ with $n\leq \infty$. For
$m=1$, the probabilities $P(Y_t\in \cdot |X_{1:n})$ are
traditionally called {\it smoothing probabilities}, when $t<n$, {\it
filtering probabilities}, when $t=n$ and {\it predictive
probabilities}, when $t>n$. In our paper, we deal with probabilities
$P(Y_{t:t+m-1} \in \cdot|X_{s:n})$, where $m\geq 1$ and $t\leq n$,
and we call all these distributions ($m$-block) smoothing
distributions even if $t=n$ or $t+m>n$. Our first main result
(Theorem \ref{th:forgetting} below) states that when $Z$ is positive
Harris chain, then  under some additional conditions, stated as {\bf A1} and {\bf A2}, the following
holds: there exists a constant $\alpha \in (0,1)$ such that for
every $ t \geq s \geq l \geq 1$, it holds
\begin{equation}\label{essa}
\sup_{n\geq t}\sup_{m\geq 1}\| P(Y_{t:t+m-1} \in \cdot|X_{l:n})-P(Y_{t:t+m-1} \in \cdot|X_{s:n}) \|_{\rm TV} \leq C_s \alpha^{t-s}=K_s \alpha^t ,
\quad \mbox{a.s.},
\end{equation}
where $C_s$ is a $\sigma(X_{s:\infty})$-measurable finite random
variable, $K_s\DEF C_s\alpha^{-s}$ and for any signed measure $\xi$ on $\Y$, $\|\xi\|_{\rm TV}
\DEF \sum_{i \in \Y}|\xi(i)|$ denotes the total variation norm of
$\xi$. Here and in what follows, when not stated otherwise, a.s. statements are with respect to measure ${\bf P}$. In this case, the distribution of $Z_1$ is not specified. Sometimes we would like to specify it, like $Z_1\sim \pi$, and then we write $P_{\pi}$-a.s. instead.
In words, (\ref{essa}) states that the total variation distance of two smoothing
distributions decrease exponentially in $t$. In Subsection \ref{subsec:ex}, we shall see that a martingale convergence argument allows us to deduce
from  (\ref{essa}) the following bound (the inequality
(\ref{vater}) below)
\begin{equation*}
\|P(Y_{t:\infty}\in \cdot|X_{l:\infty})-P(Y_{t:\infty}\in
\cdot|X_{s:\infty}) \|_{\rm TV} \leq K_s \alpha^{t}, \quad
\mbox{a.s.}.
\end{equation*}
We also argue that the same approach (and the same assumptions)
yields to the inequality
\begin{equation}\label{tilde0}
\|P_{\pi}(Y_{t:\infty}\in
\cdot|X_{s:n})-P_{\tilde{\pi}}(Y_{t:\infty}\in \cdot|X_{s:n})\|_{\rm
TV}\leq K_s\alpha^{t},\quad  P_{\pi}-{\rm a.s.}.
\end{equation}
where $\pi$ and $\tilde{\pi}$ are two initial distributions of $Z_1$
respectively, $\tilde{\pi}$ is absolutely continuous with respect to $\pi$, denoted by $\tilde{\pi}\succ \pi$, $P_{\pi}$ and
$P_{\tilde{\pi}}$ are the distributions of $Z$ under $\pi$ and
$\tilde{\pi}$ and, as previously, $\infty\geq n \geq t \geq s \geq
1$. Since $K_s$ is $P_{\pi}$-a.s. finite, the inequality (\ref{tilde0})
implies that for $P_{\pi}$-a. e. realization of $X$, the difference
$\|P_{\pi}(Y_{t:\infty}\in
\cdot|X_{s:\infty})-P_{\tilde{\pi}}(Y_{t:\infty}\in
\cdot|X_{s:\infty})\|_{\rm TV}$ tends to zero exponentially fast in
$t$. The convergence to zero is sometimes referred to as {\it the
weak ergodicity of Markov chain in random environment} \cite{vH09b},
we thus prove that the weak ergodicity is actually geometric.
Although (\ref{tilde0}) implies (\ref{essa}), in the present paper
we concentrate on the inequalities of type (\ref{essa}), because
they allow us to obtain the two sided generalizations. For two-sided
versions of these inequalities, let us consider the two-sided stationary Markov
chain $\{Z_k\}_{k \in \mathbb{Z}}$. In Subsection \ref{subsec:ex}, we shall see that
$$\lim_{l,s \rightarrow \infty} P(Y_{t:t+m-1}\in \cdot|X_{t-l:t+s})
= P(Y_{t:t+m-1}\in \cdot|X_{-\infty:\infty}),\quad {\rm a.s.}.$$
 We strengthen this result by
proving that under general conditions the following holds
(Corollary {\ref{korra2}): there exists $\alpha \in (0,1)$ and a stationary
process $\{C_k\}_{k \in \mathbb{Z}}$, $C_k< \infty$,
such that for all $t \in \mathbb{Z}$, $m\geq 1$, and $l,s \geq 0$
\begin{align}\label{kkk}
\|P(Y_{t:t+m-1}\in \cdot|X_{t-l:t+s})-P(Y_{t:t+m-1}\in \cdot|X_{-\infty:\infty})\|_{\rm TV} \leq C_t
\alpha^{l \wedge s}, \quad \mbox{a.s.},
\end{align}
where $\wedge$ denotes the minimum. The random variable $C_t$ is $\sigma(X_{-\infty:t},X_{t+m-1,\infty})$-measurable and
the process $\{C_k\}_{k \in \mathbb{Z}}$ is  ergodic when
 $\{Z_k\}_{k \in \mathbb{Z}}$ is ergodic. Another result of this type
 (Corollary \ref{korra3} below) states that under the same assumptions (we
take $m=1$, for simplicity)
\begin{equation}\label{kk}
\|P(Y_{t}\in \cdot|X_{1:n})-P(Y_{t}\in
\cdot|X_{-\infty:\infty})\|_{\rm TV}\leq C_1\alpha^{t-1}+\bar{C}_k{\alpha}^{k-t} ,\quad {\rm a.s.},
\end{equation}
where $C_1$ is $\sigma(X_{1:\infty})$-measurable, $\bar{C}_k$
is $\sigma(X_{-\infty:k})$-measurable and the process $\{\bar{C}_k\}_{k\in
\mathbb{Z}}$ is ergodic when $\{Z_k\}_{k \in
\mathbb{Z}}$ is. Although the constants $C$ and $\bar{C}$ in all
above-stated inequalities are random, nevertheless the  bounds can
be useful in various situations when pathwise limits are of
interest. For example, the inequality (\ref{tilde0}) implies that
$$\limsup_t \ln {1\over t} \sup_{n\geq t}\|P_{\pi}(Y_{t:\infty}\in
\cdot|X_{1:n})-P_{\tilde{\pi}}(Y_{t:\infty}\in
\cdot|X_{1:n})\|_{\rm TV}\leq \ln \alpha<0$$
 (for similar type of bounds see also
\cite{Ung2,forgInit3,LeGlandMevel})  and the inequality (\ref{kkk})
is very useful  when one needs to approximate the smoothing
probability $P(Y_{t:t+m-1}\in \cdot|X_{t-l:t+s})$ with something
being independent of $l$ and $s$. We shall briefly discuss the
motivation of inequalities type (\ref{kk}) and (\ref{kkk}) in the
point of view of segmentation theory  below. The  assumptions {\bf A1} and {\bf A2} are stated, discussed and interpreted in Subsection \ref{subsec:assump}.
\paragraph{Relation with the previous work.}
The most popular type of
PMM's are HMM's, where  the underlying process $Y$ is a Markov
chain, and given $Y_n=i$, the observation $X_n$ is generated
according to a probability distribution attached to the state $i$
and independently of everything else. Therefore, the vast
majority of the study of smoothing and filtering probabilities are
done for HMM's, where the study of these issues has relatively long
history dating back to  1960's, where well-known forward-backwards
recursions for calculating these probabilities (for HMM's) were
developed. The {\it forgetting properties} typically refer to the
convergence
\begin{equation}\label{flt}
\|P_{\pi}(Y_{t}\in \cdot|X_{s:n})-P_{\tilde{\pi}}(Y_{t}\in
\cdot|X_{s:n})\|_{\rm TV}\to 0,\quad {\rm a.s.}\end{equation} (as
$t\to \infty$, $n\geq t$) and the inequalities of type
(\ref{tilde0}) are often referred to as {\it exponential smoothing}.
For $n=t$, the convergence (\ref{flt}) is called {\it filter
stability}, and it is probably the most studied convergence in the
literature. For an overview of several forgetting properties and
mixing type conditions ensuring forgetting (in HMM case), we refer
to \cite[Ch. 3,4]{HMMbook}. Some of these conditions are also
restated in Subsection \ref{HMM}. The list of research articles
dealing with various aspects of forgetting and filtering problems in
HMM setting is really long including
\cite{ChigSPA,ChigFilter,forgInit1,forgInit3,DMR,LeGlandMevel,Ung3,Ung2,forgInit2,zeitouni,smoothing},
just to mention a few more prominent articles.
Majority of
above-mentioned papers deal with (exponential) forgetting of filters
and filter stability of the state space models i.e. they consider
the case where the state space $\Y$ of Markov chain is very general, possibly uncountable.
In these papers, various mixing conditions for filter stability and forgetting properties are stated. These conditions are often appropriate  and justified for general $\Y$, but
when applied to the case of finite $\Y$, they  might become  too restrictive or limited.
Hence the case of finite $\Y$ needs special treatment and so it is also quite expected that in the case of finite
  $\Y$
 our main assumption {\bf A1}, designed for discrete $\Y$, is more general that the ones made in all
above-mentioned papers. For many  models
mentioned in the literature, {\bf A1} is easy to verify, but
 we provide a more practical  condition -- {\it cluster
condition} -- which is more general than many similar assumptions
encountered in the HMM-literature, yet very easy to check.
Since HMM's are so
important class of models,  Subsection \ref{HMM} is fully devoted to
HMM-case.
Besides presenting the results,  Subsection \ref{HMM} also aims to give a state-of-art
overview of mixing-type conditions for finite-state HMM's.
\\\\
Recently, a significant contribution to the study of smoothing
probabilities (with continuous state space) was made by  van Handel
and his colleagues
\cite{vH08,vH09a,vH09b,vH09c,ChigvH10,vH12,TvH12,TvH14,vH15}. Again,
most of the papers deals with HMM's, but in  \cite{TvH12,TvH14},
also more general PMM's are considered.  In particular, they
consider a special class of PMM's, called non-degenerate PMM's. The
crucial feature of non-degenerate PMM's is that by some change of
measure the dynamics of $X$ and $Y$-process can be made independent
(see Subsection \ref{sec:nondeg} for precise definition).  While
natural in continuous-space setting, for finite $\Y$ this assumption
might be restrictive and in Subsection \ref{sec:nondeg} we show that
the assumptions in \cite{TvH12} imply {\bf A1} and {\bf A2}. For
HMM's, the non-degeneracy simply means strictly positive emission
densities and in Subsection \ref{HMM} we show several ways how to
relax it.
\\\\
The present work generalizes and builds on the approach in
\cite{smoothing,smoothing2}, where solely the HMM-case was
considered. In many ways,  HMM  is technically much simpler model to
handle, hence the generalization from HMM to PMM is far from being
straightforward. Moreover, our second main result, Theorem
\ref{thm2} cannot be found in the earlier papers even in HMM case.
Also, for the HMM case, the cluster condition introduced in the
present paper is significantly weaker than the one in
\cite{smoothing,smoothing2}. The proofs of our main results rely on
the Markovian block-decomposition of the conditional hidden chain,
 \textbf{A1} is used to bound from above
the Dobrushin coefficient of certain block-transitions.
\paragraph{Applications in segmentation.} The motivation of studying the inequalities (\ref{essa}) and the
two-sided inequalities (\ref{kkk}) and (\ref{kk}) (instead of just
filtering ones) comes from the so-called {\it segmentation problem}
that aims to prognose or estimate the hidden underlying path
$y_{1:n}$ given a realization $x_{1:n}$ of observed process
$X_{1:n}$. The goodness of any path $s_{1:n}\in {\cal Y}^n$ is typically measured via
 {\it loss function} $L: {\cal Y}^n\times {\cal Y}^n \to [0,\infty],$ where
$L(y_{1:n},s_{1:n})$ measures the loss when the actual state sequence is
$y_{1:n}$ and the estimated sequence is $s_{1:n}$. The best path is then the one that minimizes the expected loss
\[
E[L(Y_{1:n},s_{1:n})|X_{1:n}=x_{1:n}]=\sum_{y_{1:n}\in
{\cal Y}^n}L(y_{1:n},s_{1:n})P(Y_{1:n}=y_{1:n}|X_{1:n}=x_{1:n}).
\]
over all possible state sequences $s_{1:n}$. A common loss function measures the similarity of the sequences entry-wise, i.e.
$$L(y_{1:n},s_{1:n})=\sum_{t=1}I_{y_t\ne s_t},$$
 where $I_{y_t\ne s_t}=0$ if and only if $y_t=s_t$, otherwise $I_{y_t\ne s_t}=1$. Thus $L(y_{1:n},s_{1:n})$ counts the classification errors  of path $s_{1:n}$ and the expected number of classification errors is
$$E[L(Y_{1:n},s_{1:n})|X_{1:n}=x_{1:n}]=n-\sum_{t=1}^nP(Y_t=s_t|X_{1:n}=x_{1:n}].$$
Now it is clear that the path $\hat{y}_{1:n}$ that minimizes the expected loss is also the path that minimizes the expected number of classification errors and it can obtained by pointwise maximization of smoothing probabilities, i.e.
\begin{align*}
\hat{y}_{1:n}&=\arg\min_{s_{1:n}}E[L(Y_{1:n},s_{1:n})|X_{1:n}=x_{1:n}]=\arg\max_{s_{1:n}}\sum_{t=1}^nP(Y_t=s_t|X_{1:n}=x_{1:n}]\quad \Leftrightarrow\quad \\
\hat{y}_t&=\arg\max_{y\in \Y}P(Y_t=y|X_{1:n}=x_{1:n}),\quad t=1,\ldots,n.\end{align*}
Any such path is called {\it pointwise
maximum a posteriori} (PMAP)
(see, e.g. \cite{seg,kuljus,peep}). The PMAP path is easy to calculate via forward-backward algorithms that hold for PMM as well as for HMM. \\\\
When $n$ varies, it is convenient to divide the expected loss by $n$, and so we study the
time-averaged expected number of
classification errors of PMAP path as follows:
$$1-{1\over n}\sum_{t=1}^n \max_{y\in \Y} P(Y_t=y|X_{1:n}).$$
This number can be considered as the (best possible) expected number of classification errors per one time entry. It turns out that when $Z$ is an ergodic process
satisfying our general assumptions {\bf
A1} and {\bf A2}, then there exists a constant $R\geq 0$ so that
$$1-{1\over n}\sum_{t=1}^n \max_{y\in \Y} P(Y_t=y|X_{1:n})\to
R,\quad \text{a.s.},$$ where $R$ is a constant. The number
$R$ is solely depending on the model and characterizes the its
segmentation capacity -- the smaller $R$, the easier the
segmentation. The proof of this
convergence in HMM-case is given in \cite{smoothing2,kuljus}, but it
holds without changes in more general PMM case as well. The proof relies largely   on the
inequality (\ref{kk}), being thus an example of the use of this kind of inequalities.
 For the discussion about the importance of the existence of limit $R$ as well
as for another applications of inequalities of type (\ref{kk}) and
(\ref{essa}) in the segmentation context, see \cite{kuljus,peep}.
These papers deal with HMM's only, but with the exponential
forgetting results of the present paper, the generalization to PMM
case is possible.
\\\\
It is interesting to notice that our main  condition \textbf{A1} is
not only relevant to smoothing distributions. This condition has
been used, albeit in slightly more restricted form, in the
development of the Viterbi process theory \cite{PMMinf,AV,Eng}. This
suggests that the condition {\bf A1} is essential form many
different aspects and captures well the mixing properties. When the
observation space $\X$ is finite, then \textbf{A1} essentially
becomes what is known in ergodic theory as the
\textit{subpositivity} of some observation string.

\section{Preliminaries}
\paragraph{The model and some basic notation.}
We will now state the precise theoretical framework of the paper. We
assume that observation-space $\mathcal{X}$ is a Polish (separable
completely metrizable) space equipped with its Borel $\sigma$-field
$\B(\X)$. We denote $\mathcal{Z}= \X \times \Y$, and equip $\Z$ with
product topology $\tau \times 2^\Y$, where $\tau$ denotes the
topology of $\X$. Furthermore, $\Z$ is equipped with its Borel
$\sigma$-field $\B(\Z)=\B(\X) \otimes 2^\Y$, which is the smallest
$\sigma$-field containing sets of the form $A \times B$, where $A
\in \B(\X)$ and $B \in 2^\Y$. Let $\mu$ be a $\sigma$-finite measure
on $\B(\X)$ and let $c$ be the counting measure on $2^\mathcal{Y}$.
Finally, let
\begin{align*}
q \colon \Z^2 \rightarrow [0,\infty ), \quad (z',z) \mapsto q(z'|z)
\end{align*}
be a  measurable non-negative function such that for each $z \in
\mathcal{Z}$ the function $z' \mapsto q(z'|z)$ is a probability density function  with
respect to product measure $\mu\times c$. We define random process
$Z =\{Z_k \}_{k\geq 1} = \{(X_k,Y_k) \}_{k \geq 1}$ as a homogeneous
Markov chain on the two-dimensional space $\mathcal{Z}$ having the
transition kernel density $q(z'|z)$. This means that the transition
kernel of $Z$ is defined as follows:
\begin{align}\label{deffa}
&P(Z_2 \in C|Z_1=z)= \int_C q(z'|z) \, \mu \times c (dz'), \quad z
\in \mathcal{Z}, \quad C \in \B(\Z).
\end{align}
Since every $C\subset \B(\Z)$ is in the form $C=\cup_{j\in
\Y}A_j\times \{j\}$, where $A_j\in \B(\X)$, the probability
(\ref{deffa}) reads
$$P(Z_2 \in C|X_1=x,Y_1=i)=\sum_{j\in \Y}P(X_2\in A_j,Y_2=j|X_1=x,Y_1=i)=\sum_{j}\int_{A_j}
q(x',j|x,i)\mu(dx').$$ We also assume that $Z_1$ has density with
respect to product measure $\mu\times c$. Then, for every $n$, the
random vector $Z_{1:n}$ has a density with respect to the measure
$(\mu\times c)^n$.  In what follows, with a slight abuse of notation
the letter $p$ will be used to denote the various joint and
conditional densities. Thus $p(z_k)=p(x_k,y_k)$ is the density of
$Z_k$ evaluated at $z_k=(x_k,y_k)$, $p(z_{1:n})=p(z_1)\prod_{k=2}^n
q(z_k|z_{k-1})$ is the density of $Z_{1:n}$ evaluated at $z_{1:n}$,
$ p(z_{2:n} |z_1)= \prod_{k=2}^n q(z_k|z_{k-1})$ stands for the
conditional density and so on. Sometimes it is convenient to use
other symbols beside $x_k,y_k,z_k$ as the arguments of some density;
in that case we indicate the corresponding probability law using the
equality sign, for example
\begin{align*}
p(x_{2:n}, y_{2:n}|x_1=x,y_1=i)=q(x_2,y_2|x,i)\prod_{k=3}^n q(x_k,
y_k|x_{k-1},y_{k-1}), \quad n \geq 3.
\end{align*}
The notation $P_z(\cdot)$ will represent the probability measure,
when the initial distribution of $Z$ is the Dirac measure on $z \in
\mathcal{Z}$ (i.e. $P_z(A)=P(A|Z_1=z)$). For a probability measure
$\nu$ on $\B(\Z)$, $P_\nu(\cdot)$ denotes the probability measure,
when the initial distribution of $Z$ is $\nu$ (i.e. $P_\nu(A)=\int
P_z(A) \, \nu(dz)$).
\\\\
The marginal processes $\{X_k \}_{k \geq 1}$ and $\{Y_k \}_{k \geq
1}$ will be denoted with $X$ and $Y$, respectively. It should be
noted that even though $Z$ is a Markov chain, this doesn't
necessarily imply that either of the marginal processes $X$ and $Y$
are Markov chains. However, it is not difficult to show that
conditionally given $X_{1:n}$, $Y_{1:n}$ is a (generally
non-homogeneous) Markov chain and vice-versa.
\\\\
For any set $A$ consisting of vectors of length $r>1$ we adopt the
following notation:
\begin{align*}
&A_{(k)}\DEF\{x_{k} \:| \: x_{1:r} \in A\}, \quad 1 \leq k \leq r.
\end{align*}
Alternatively, $A_{(k)}=f_k(A)$, where $f_k$ is the $k$-th
projection.
\paragraph{General state space Markov chains.} We will now recall some necessary concepts from the general state Markov chain theory. Markov chain $Z$ is called
\textit{$\varphi$-irreducible} for some $\sigma$-finite measure
$\varphi$ on $\B(\Z)$, if $\varphi(A)>0$ implies $\sum_{k=2}^\infty
P_z(Z_k \in A)>0$ for all $z \in \Z$. If $Z$ is
$\varphi$-irreducible, then there exists \cite[Prop. 4.2.2.]{MT} a
\textit{maximal irreducibility measure} $\psi$ in the sense that for
any other irreducibility measure $\varphi'$ the measure $\psi$
dominates $\varphi'$, $\psi \succ \varphi'$. The symbol $\psi$ will
be reserved to denote the maximal irreducibility measure of $Z$.
Chain $Z$ is called \textit{Harris recurrent}, when it is
$\psi$-irreducible and $\psi(A)>0$ implies $P_z(Z_k \in A \mbox{
i.o.})=1$ for all $z \in \Z$. Chain $Z$ is called \textit{positive}
if its transition kernel admits an invariant probability measure.
Any $\psi$-irreducible chain admits a cyclic decomposition \cite[Th.
5.4.4]{MT}: there exists disjoint sets $D_0,\ldots,D_{d-1} \subset
\Z$, $d\geq 1$, such that
\begin{enumerate}[label=(\roman*)]
\item for $z \in D_k$, $P_z(Z_2 \in D_{k+1})=1$, $k=0,\ldots,d-1 \pmod d$;
\item $(\cup_{k=1}^d D_k )^{\mathsf{c}}$ is $\psi$-null.
\end{enumerate}
The cycle length $d$, called the {\it period} of $Z$, is chosen to
be the largest possible in the sense that for any other collection
$\{d',D_k', k=0,\ldots, d'-1\}$ satisfying (i) and (ii), we have
$d'$ dividing $d$; while if $d=d'$, then, by reordering the indices
if necessary, $D_k'=D_k$ a.e. $\psi$. A $\psi$-irreducible chain $Z$
is called \textit{aperiodic}, when its period is 1, $d=1$.
\paragraph{Overlapping $r$-block process.} For every $r>1$, define
 ${\bf Z}_{k}\DEF Z_{k:k+r-1}$, $k \geq 1$. Thus $\ZZ=\{\ZZ_k\}$ is a
Markov process with the state space  ${\cal Z}^r$ and transition
kernel
$$P(\ZZ_{2}\in A|\ZZ_1=z_{1:r})=P\big(Z_{2:r+1}\in
A|Z_{1:r}=z_{1:r}\big)=P\big( Z_{r+1}\in
A(z_{2:r})|Z_1=z_1),$$ where  $A\in {\cal B}(\Z)^r$,
$$A(z_{2:r})\DEF\{z: (z_{2:r},z)\in A\}.$$
Similarly, for every set $A\subset \Z^r$, and $z_1\in \Z$, we denote
 $A(z_1)\DEF\{z_{2:r}\:| \: z_{1:r} \in A\}$. The
following proposition (proof in Appendix) specifies the maximal
irreducible measure of $\ZZ$ .
\begin{proposition}\label{block} If $Z$ is positive Harris with stationary
probability measure $\pi$ and maximal irreducible measure $\psi$,
then $\ZZ$ is a positive Harris chain with maximal irreducible
measure $\psi_r$, where
\begin{equation}\label{moot}
{\mathbf \psi}_r(A)\DEF\int_{A_{(1)}}P\big(Z_{2:r}\in
A(z_1)\mid Z_1=z_1\big)\psi(dz_1),\quad A\in \B(\Z)^{\otimes
r}.\end{equation}
\end{proposition}
\section{Exponential forgetting}
\subsection{The main assumptions}\label{subsec:assump} We shall now introduce the basic assumptions of our theory
for the non-stationary case. For every $n \geq 2$ and $i,j \in \Y$
we denote
\begin{align*}
p_{ij}(x_{1:n}) \DEF {\sum_{y_{2:n} \colon y_n=j}
p(x_{2:n},y_{2:n}|x_1,y_1=i)=p(x_{2:n},y_n=j|x_1,y_1=i)}.
\end{align*}
For any $n \geq 2$, define
\begin{align} \label{Ypairs}
\Y^+(x_{1:n}) \DEF \{ (i,j) \: | \: p_{ij}(x_{1:n})>0\}, \quad
x_{1:n} \in \X^n.
\end{align}
Recall the definition of $A_{(k)}$ and $A(x_1)$. Thus
$$\Y^+(x_{1:n})_{(1)}=\{i \: | \: \exists j \text{   such that }
p_{ij}(x_{1:n})>0\},\quad \Y^+(x_{1:n})_{(2)}=\{j \: | \: \exists i \text{
  such that } p_{ij}(x_{1:n})>0\}.$$ Observe that it is not generally the case that $\Y^+(x_{1:n})=\Y^+(x_{1:n})_{(1)} \times \Y^+(x_{1:n})_{(2)}$.
The following are the main assumptions.
\begin{description} \item{\textbf{A1}} There exists integer  $r>1$ and a set $E \subset \X^{r}$
such that $\Y^+ \DEF \Y^+(x_{1:r}) \neq \emptyset$ is the same
for all $x_{1:r} \in E$, and $\Y^+=\Y^+_{(1)} \times
\Y^+_{(2)}$.

\item{\textbf{A2}}
Chain $Z$ is $\psi$-irreducible, with $\psi(E_{(1)} \times
\Y^+_{(1)})>0$. Furthermore, $\mu^{r-1}(E(x_{1}))>0$ for all
$x_1 \in E_{(1)}$.
\end{description}
The condition \textbf{A1} is the central assumption of our theory.
The intuitive meaning of {\bf A1} is fairly
simple, because it can be considered as the "irreducibility and
aperiodicity" of conditional signal process as follows. Suppose we
have an inhomogeneous Markov chain $Y=\{Y_t\}_{t \geq 1}$, with
$\Y_t$ being the finite state space of $Y_t$. The canonical concepts
of irreducibility and aperiodicity are not defined for such a Markov
chain, but a natural generalization would be the following: for
every time $t$, there exists a time $n>t$ such that
$P(Y_n=j|Y_t=i)>0$ for every $i\in \Y_t$ and $j\in \Y_n$. If $Y$ is
homogeneous, then this property implies that $Y$ is irreducible and
aperiodic, hence also geometrically ergodic. When we fix $n>t$ and define
$$\Y^+=\{(i,j): i\in \Y_t,j\in \Y_n, P(Y_n=j|Y_t=i)>0\},$$
then the above-stated condition reads $\Y^+=\Y_t\times \Y_n$. The assumption {\bf A1}
generalizes that idea to conditional signal process. Indeed, {\bf A1} states that for every $x_{1:r}\in E$, and for every fixed $t\geq 1$, it holds that
$\Y^+=\Y^+_{(1)}\times \Y^+_{(2)}$, where  $$\Y^+=\{(i,j)\in \Y^2:
P(Y_{t+r-1}=j|Y_t=i, X_{t:t+r-1}=x_{1:r})>0\}$$
and
\begin{align*}
\Y^+_{(1)}&=\{i\in \Y: \exists j\in \Y,\quad \text{such that  }  P(Y_{t+r-1}=j|Y_t=i,\quad
X_{t:t+r-1}=x_{1:r})>0\}\\
 \Y^+_{(2)}&=\{j\in \Y: \exists i\in \Y \quad \text{such that  }
P(Y_{t+r-1}=j|Y_t=i,\quad
X_{t:t+r-1}=x_{1:r})>0\}.\end{align*}
Observe that since $Z$ is homogenous,  the set $\Y^+(x_{1:r})$ (and therefore also the sets $\Y^+_{(1)}$ and $\Y^+_{(2)}$) is independent of $t$, and {\bf A1} also ensures that it is independent of $x_{1:r}$, provided $x_{1:r}\in E$. If now $x_{1:\infty}$ is a realization of $X_{1:\infty}$ such that $x_{1:r}\in E$ and we define $\Y^+(x_{1:\infty})$ as previously, just $x_{1:r}$ replaced by $x_{1:\infty}$, then $\Y^+_{(1)}(x_{1:\infty})=\Y^+_{(1)}(x_{1:r})$, $\Y^+_{(2)}(x_{1:\infty})\subseteq\Y^+_{(2)}(x_{1:r})$, hence when $\Y^+(x_{1:r})=\Y_{(1)}^+(x_{1:r})\times \Y_{(2)}^+(x_{1:r})$, then also
$\Y^+(x_{1:\infty})=\Y_{(1)}^+(x_{1:\infty})\times \Y_{(2)}^+(x_{1:\infty})$. This observation makes {\bf A1} comparable with the definition of {\it irreducibility of conditional signal process} defined by van Handel in \cite{vH09b}. Van Handel's
definition, when adapted to our case of finite $\Y$, states that for every $t$ and
for a.e. realization $x_{t:\infty}$ of $X_{t:\infty}$, there exists
$n>t$ such that the measures $P(Y_n\in \cdot| Y_t=i_1,\,
X_{t:\infty}=x_{t:\infty})$ and $P(Y_n\in \cdot| Y_t=i_2,\,
X_{t:\infty}=x_{t:\infty})$ are not mutually singular, provided
$P(Y_t=i_k|X_{t:\infty}=x_{t:\infty})>0$ for $k=1,2$ (for non-Markov
case this condition is generalized in \cite{TvH14}). This condition
is weaker than {\bf A1}, and it has to be, because by Theorem 2.3 in
\cite{vH09b}, for stationary $X$, the above-defined irreducibility
condition is necessary and sufficient for the convergence
$\|P_{\pi}(Y_{t}\in \cdot|X_{1:\infty})-P_{\tilde{\pi}}(Y_{t}\in
\cdot|X_{1:\infty})\|_{\rm TV}\to 0,$ $P_{\pi}$-a.s. and $P_{\tilde
\pi}$-a.s., where $\pi\succ \tilde{\pi}$ and $\pi$ corresponds to
the stationary measure (weak ergodicity). On the other hand, the
condition {\bf A1} is typically met and relatively easy to verify.
Moreover, as already mentioned, our main result, Theorem
\ref{th:forgetting} states that for positive Harris $Z$, {\bf A1}
and {\bf A2} do ensure not only the weak ergodicity but also the
exponential rate of convergence. So one possibility to  generalize {\bf A1} for uncountable $\Y$ would be replacing "not mutually singular" in the definition of regularity of conditional signal process in \cite{vH09b} by "having the same support".
\\\\
The condition
\textbf{A2} ensures that $X$ returns to the set $E$ in appropriate
regularity under certain stability conditions on $Z$. Conditions
\textbf{A1}-\textbf{A2} will be discussed in more detail in case of
specific models in Section \ref{sec:Examples}.
\\\\
We note that under \textbf{A1} and \textbf{A2} we may without
loss of generality assume that for some $n_0 \geq 1$
\begin{align} \label{ineq:n0Bound}
\frac{1}{n_0}  \leq  p_{ij}(x_{1:r})\leq n_0, \quad \forall
(i,j) \in \Y^+, \quad \forall x_{1:r} \in E.
\end{align}
Indeed, let for $n \geq 1$ $$E_n=\left\{x_{1:r} \in E \: \middle| \:
\frac{1}{n} \leq \min_{(i,j) \in \Y^+}p_{ij}(x_{1:r}) \leq
\max_{(i,j) \in \Y^+}p_{ij}(x_{1:r}) \leq n \right\}.$$
 By
\textbf{A1} $ p_{ij}(x_{1:r})>0$ for every $x_{1:r} \in E$ and
$(i,j) \in \Y^+$, and so $E_n \nearrow E$. Define measure $\psi_0$
on $\X$ by $\psi_0(A)=\psi(A \times \Y^+_{(1)})$. Take $n_0$ so
large that $\psi_0 \times \mu^{r-1}(E_{n_0})>0$; this is possible by
\textbf{A2}. We would like to replace $E$ by $E_{n_o}$. Clearly \textbf{A1} holds for $E_{n_o}$ as well, and we also have
$$\psi_0 \times \mu^{r-1}(E_{n_0})=\int_{E_{n_0}(1)} \mu^{r-1}(E_{n_0}(x))\psi_o(dx)>0.$$
Unfortunately, the positive integral does not imply that $\mu^{r-1}(E_{n_0}(x))>0$ for every $x\in E_{n_0}(1)$, a property needed for \textbf{A2}. But surely there
exists a set  $E' \subset E_{n_0}$ such that
$\psi_0(E'_{(1)})>0$ and $\mu^{r-1}(E'(x_1))>0$ for all $x_1 \in
E'_{(1)}$ (otherwise the integral would be zero). Thus $E'$ satisfies both \textbf{A1} and \textbf{A2}, and
so with no loss of generality we may and shall assume that
\eqref{ineq:n0Bound} holds.
\subsection{Bounding the Dobrushin coefficient}\label{sec:dobrushin} The \textit{Dobrushin coefficient} $\delta(M)$ of a stochastic matrix $M(i,j)$ is defined as
the maximum total variation difference over all row pairs of $M$ divided by 2, i.e.
\begin{align*}
\delta(M)\DEF\dfrac{1}{2}\max_{1 \leq i < i' \leq n} \|M(i,
\cdot)-M(i',\cdot)\|_{\rm TV},
\end{align*}
where $\|\cdot\|_{\rm TV}$ stands for total variation norm. As is
well known, for any two probability rows vectors $\xi,\xi'$ of
length $n$, and for any $n \times n$ stochastic matrix, $$\|\xi M -
\xi' M\|_{\rm TV} \leq \delta(M) \|\xi-\xi'\|_{\rm TV} \leq 2
\delta(M).$$ The Dobrushin coefficient is sub-multiplicative: for
any two $n \times n$ stochastic matrices $M$ and $M'$,
$\delta(MM')\leq \delta(M)\delta(M')$. A stochastic matrix $M$ is
said to satisfy the \textit{Doeblin condition}, if there exists a
probability row vector $\xi$ and $\epsilon>0$ such that each row of
$M$ is uniformly greater than $\epsilon \xi$, i.e. $M(i,j) \geq
\epsilon \xi(j)$ for all $i,j \in \{1, \ldots,n\}$. If $M$ satisfies
such condition, then its Dobrushin coefficient has an upper bound
$\delta(M)\leq 1-\epsilon$.
\\\\
In what follows we prove our own version of the Doeblin condition.
We shall consider the observation sequences $x_{1:n}$, where
$x_{1:r}\in E$ and $n\geq r$. For those sequences define probability
distribution on $\Y$ as follows:
\begin{align*}
\lambda[x_{r:n}](j)\DEF\begin{cases}\dfrac{1}{c(x_{r:n})} p(x_{r+1:n}|y_{r}=j,x_{r}) \mathbb{I}_{\Y^+_{(2)}}(j), & \mbox{if $n >r$}\\
\mathbb{I}_{\Y^+_{(2)}}(j)/|\Y^+_{(2)}|, &\mbox{if $n=r$}
\end{cases},
\end{align*}
where $c(x_{r:n})\DEF \sum_{y_r \in \Y^+_{(2)}}
p(x_{r+1:n}|x_r,y_r)$ is the normalizing constant, $\mathbb{I}_A$
denotes the indicator function on $A$, and the set $\Y^+$ is given
by \textbf{A1}. Define the stochastic matrix
\begin{align*}
U[x_{1:n}](i,j)\DEF\begin{cases}\dfrac{p(y_r=j,x_{2:n}|x_1,y_1=i)}{p(x_{2:n}|x_1,y_1=i)}, & \mbox{if $p(x_{2:n}|x_1,y_1=i)>0$}\\
\lambda[x_{r:n}](j), &\mbox{if $p(x_{2:n}|x_1,y_1=i)=0$ and
$c(x_{r:n})>0$}
\end{cases},
\end{align*}
where $i,j \in \Y$ represent the row and column index of the matrix,
respectively. The matrix $U[x_{1:n}]$ is well-defined stochastic matrix for all
$x_{1:n}$ satisfying $c(x_{r:n})>0$. Furthermore, $j \mapsto
U[x_{1:n}](i,j)$ is (a regular) version of the conditional
distribution $P(Y_{k+r}=j|X_{k+1:k+n}=x_{1:n},Y_{k+1}=i)$ for all $k
\geq 0$.

\begin{lemma} \label{lem:Doeblin} Suppose \textbf{A1} is satisfied. Let $n \geq r$, and let $x_{1:n}$ be such that $x_{1:r} \in E$ and $p(x_{2:n}|x_1,y_1=i^*)>0$ for some $i^* \in \Y$. Then
\begin{align}  \label{ineq:Doeblin}
U(i,j)[x_{1:n}]\geq \frac{1}{n_0^2} \lambda[x_{r:n}](j), \quad
\forall i,j \in \Y.
\end{align}
\end{lemma}
\begin{proof} We only consider the case $n >r$; the proof for $n=r$ follows along similar, although simpler arguments. Let $x_{1:n}$ be such that $x_{1:r} \in E$ and $p(x_{2:n}|x_1,y_1=i^*)>0$ for some $i^* \in \Y$. First we show that
\begin{align} \label{equiv:px}
p(x_{2:n}|x_1,y_1)>0 \quad \mbox{if and only if} \quad  y_1 \in
\Y^+_{(1)}.
\end{align}
Indeed, since by assumption, $p(x_{2:n}|x_1,y_1=i^*)=\sum_j
p_{i^*j}(x_{1:r})p(x_{r+1:n}|x_r,y_r=j)>0$, then there exists a
$j^*$ such that
\begin{align}
&p_{i^*j^*}(x_{1:r})>0 \quad \mbox{and} \notag\\
&p(x_{r+1:n}|x_r,y_r=j^*)>0.  \label{ineq:j*}
\end{align}
Thus $(i^*,j^*) \in \Y^+$, which by \textbf{A1} implies that
$(i,j^*) \in \Y^+$ for every $i \in \Y^+_{(1)}$. This together with
\eqref{ineq:j*} shows that $p(x_{2:n}|x_1,y_1)>0$ for every $y_1 \in
\Y^+_{(1)}$, and so \eqref{equiv:px} is proved in one direction. In
the other direction, if $i \notin \Y^+_{(1)}$, then, by definition
of $\Y^+$, $p_{ij}(x_{1:r})=0$ for every $j \in \Y$, which in turn
implies that $p(x_{2:n}|x_1,y_1=i)=0$.
\\\\
Next, note that $c(x_{r:n})>0$, and so $U[x_{1:n}]$ is well-defined.
Indeed, we saw above that $j^* \in \Y^+_{(2)}$ and therefore by
\eqref{ineq:j*} $c(x_{r:n}) \geq p(x_{r+1:n}|x_r,y_r=j^*) >0.$ When
$i \notin \Y^+_{(1)}$ then by \eqref{equiv:px}, $U[x_{1:n}](i,j)=
\lambda[x_{r:n}](j)$ for every $j\in \Y$ and hence the inequality
\eqref{ineq:Doeblin} is fulfilled for every $j \in \Y$. Thus in what
follows we assume that $i \in \Y^+_{(1)}$. We have
\begin{align*}
U[x_{1:n}](i,j)&={p(y_{r}=j,x_{2:n}|y_1=i,x_1)\over p(x_{2:n}|x_1,y_1=i)}\\
&={p(y_{r}=j,x_{2:r}|x_1,y_1=i)p(x_{r+1:n}|y_{r}=j,x_{r})\over p(x_{2:n}|y_1=i,x_1)}\\
&={p_{ij}(x_{1:r})p(x_{r+1:n}|y_{r}=j,x_{r})\over {\sum_{j' \in \Y}
p_{i j'}(x_{1:r})p(x_{r+1:n}|y_r=j',x_r)}}.
\end{align*}
By \textbf{A1} $p_{ij}(x_{1:r})>0$ if and only if $j \in
\Y^+_{(2)}$, and so we obtain
\begin{align*}
U[x_{1:n}](i,j)&={p_{ij}(x_{1:r})p(x_{r+1:n}|y_{r}=j,x_{r})
\mathbb{I}_{\Y^+_{(2)}}(j)\over {\sum_{j' \in \Y^+_{(2)}} p_{i
j'}(x_{1:r})p(x_{r+1:n}|y_r=j',x_r)}}.
\end{align*}
Together with \eqref{ineq:n0Bound} this implies
\begin{align*}
U[x_{1:n}](i,j)\geq  \dfrac{1}{n_0^2 \cdot c(x_{1:r})}
p(x_{r+1:n}|y_{r}=j,x_{r})
\mathbb{I}_{\Y^+_{(2)}}(j)={1\over n_o^2}\lambda[x_{1:n}](j).
\end{align*}
\end{proof}\\\\
{\bf Remark.} Inspired by the technique  in \cite{DMR}, one can add
to the condition {\bf A1}  the  following: there exists $t\in
\{2,\ldots,r-1\}$ and state $l\in \Y$ so that for every $x_{1:r}\in
E$ \begin{equation}\label{sopot} {p_{il}(x_{1:t})p_{l
j}(x_{t:r})\over p_{ij}(x_{1:r})}>0,\quad \forall i,j \in
\Y^+_{(1)}\times \Y^+_{(2)}.\end{equation}
Then, for every $i\in
\Y^+_{(1)}$,
\begin{align*}
{p(y_t=l,x_{2:n}|x_1,y_1=i)\over p(x_{2:n}|x_1,y_1=i)}
&={\sum_{j\in \Y^+_{(2)}}p_{il}(x_{1:t})p_{lj}(x_{t:r})p(x_{r+1:n}|y_r=j,x_r)\over \sum_{j' \in \Y_{(2)}} p_{ij'}(x_{1:r})p(x_{r+1:n}|x_r,y_r=j')}\\
&={{\sum_{j\in \Y^+_{(2)}}{p_{il}(x_{1:t})p_{lj}(x_{t:r})\over
p_{ij}(x_{1:r})} p_{ij}(x_{1:r})p(x_{r+1:n}|y_{r}=j,x_{r}) \over
\sum_{j' \in \Y_{(2)}} p_{i j'}(x_{1:r})p(x_{r+1:n}|x_r,y_r=j')}}\\
&\geq \min_{i,j\in \Y^+_{(1)}\times \Y^+_{(2)}}{p_{il}(x_{1:t})p_{l
j}(x_{t:r})\over p_{ij}(x_{1:r})}\DEF \lambda.\end{align*} Thus the
matrix of conditional probabilities
$V(i,j)=P(Y_{k+t}=j|X_{k+1:k+n}=x_{1:n},Y_{k+1}=i)$ could be defined
so that it satisfies: $V(i,l)\geq \lambda $, for every $i\in \Y$.
Thus, the Dobrushin condition holds with
$\lambda(j)=\mathbb{I}_{\{l\}}(j)$. Although formally (\ref{sopot})
restricts {\bf A1}, for many models like HMM, it is actually
equivalent to {\bf A1}. One advantage of (\ref{sopot}) is that
$\lambda$ might be bigger than $1/n^2_o$. The condition
(\ref{sopot}) is more useful in linear state space models
(continuous $\Y$), see \cite{DMR}.
\subsection{Exponential forgetting results}\label{subsec:ex}
\paragraph{Conditional transition matrices and distributions.} Let now, for every $m\geq 1$ and for every $k\geq 1$,
$$F_{k;m}[x_{1:n}]=(F_{k,m}[x_{1:n}](u,v))_{u\in \Y,v \in \Y^{m}}$$ be
the $|\Y|\times |\Y|^m$-matrix  such that
$$F_{k;m}[x_{1:n}](u,v)\DEF P(Y_{k+1:k+m}=v|X_{1:n}=x_{1:n},Y_{1}=u).$$
Observe that $\Y^m$ is countable and so every version of conditional
probability above is regular. Note that since the process $Z$ is
homogeneous, for every $1<s<n$ and $x_{1:n}$, we can take
\begin{equation}\label{abi}
F_{k;m}[x_{s:n}](u,v)=P(Y_{s+k:s+k+m-1}=v|X_{s:n}=x_{s:n},Y_{s}=u)=P(Y_{s+k:s+k+m-1}=v|X_{1:n}=x_{1:n},Y_{s}=u),
\end{equation}
where the last equality follows from Markov property.  Thus
$$F_{1;1}[x_{s:n}](u,v)=P(Y_{s+1}=v\mid X_{s:n}=x_{s:n},Y_s=u)=P(Y_{s+1}=v\mid X_{1:n}=x_{1:n},Y_s=u)$$ is the
one-step conditional transition matrix. The matrix
$F_{0;1}[x_{1:n}]\equiv I$, where $I$ stands for $|\Y|\times |\Y|$
identity matrix, and for $m>1$,
\begin{equation}\label{0m}
F_{0;m}[x_{1:n}](u,v)\DEF\left\{
                           \begin{array}{ll}
                             0, & \hbox{if $u\ne v_1$;} \\
                             F_{1;m-1}[x_{1:n}](v_1,v_{2:m}) & \hbox{if $u=v_1$.}
                           \end{array}
                         \right.
\end{equation}
The definition (\ref{0m}) is clearly justified, since if $m>1$ and
$u=v_1$, then
$$P(Y_1=v_1,Y_{2:m}=v_{2:m}|X_{1:n}=x_{1:n},Y_1=u)=P(Y_{2:m}=v_{2:m}|X_{1:n}=x_{1:n},Y_1=v_1).$$
Note that without loss of generality we may assume
$F_{r-1;1}[x_{1:n}] \equiv U[x_{1:n}]$.\\\\
For every $m\geq 1$ and $1\leq l\leq t,n$ define
$$
\nu^t_{l:n;m}[x_{l:n}](v)\DEF P(Y_{t:t+m-1} = v| X_{l:n}=x_{l:n}),\quad
v\in \Y^m. $$  The notation $\nu^t_{l:n;m}$ represents the random
function $\nu^t_{l:n;m}[X_{l:n}]$ taking values $[0,1]^{m}$. The domain of that function is finite and so we identify the random function with random vector. Observe that for
any $m, k,l, \geq 1$, $s \geq l $ and $n \geq s+k$ it holds
\begin{align*}
\nu^{s+k}_{l:n;m}(v)&=\mathbb{E}[P(Y_{s+k:s+k+m-1} =v | X_{l:n},Y_s)| X_{l:n}]=\mathbb{E}[F_{k;m}[X_{s:n}](Y_s,v)| X_{l:n}]\\
&=\sum_{u\in \Y} F_{k;m}[X_{s:n}](u,v)\nu^s_{l:n;1}[X_{l:n}](u)=(\nu^s_{l:n;1}F_{k;m}[X_{s:n}])(v),
\end{align*}
where the second equality follows from (\ref{abi}), and the third
equality follows from the fact that $\nu^s_{l:n;1}$ is a regular
conditional distribution. Thus
\begin{equation}\label{decomp}
\nu^{s+k}_{l:n;m}=\nu^s_{l:n;1}F_{k;m}[X_{s:n}],\quad{\rm a.s.}.
\end{equation}
In order to generalize the $\nu^t_{l:n;m}$ to the case $m=\infty$
corresponding to the conditional distribution of $Y_{t:\infty}$, let
${\cal F}$ stand for the cylindrical $\sigma$-algebra on
$\Y^{\infty}$. Now, for every $1\leq l\leq t,n$, let
$\nu^t_{l:n;\infty}$ be the regular version of the conditional
distribution $P(Y_{t:\infty}\in \cdot|X_{l:n})$ on $\sigma$-algebra
${\cal F}$.
\paragraph{Remark about a.s.}
The stochastic process $Z$ is defined on an underlying probability space $(\Omega, {\cal F}, {\bf P})$. (Regular) conditional probabilities  $\nu$ and $F$ are defined up to ${\bf P}$-a.s., only. Therefore, the (in)equalities like  (\ref{decomp}) or the statement (\ref{kpa}) of  Proposition \ref{prop:TVbound} below are all stated in terms of ${\bf P}$-a.s.. Observe also that there are countable many indexes $l,n,m,s,k$. Therefore (\ref{decomp}) implies: there exists $\Omega_o\subset \Omega$ such that ${\bf P}(\Omega_o)=1$ and for any $\omega\in \Omega_o$
(\ref{decomp}) holds for any $s,k,l,m,n$. The same holds for other similar equalities like (\ref{kpa}).
\paragraph{The main theorem.}
In what follows, we take $r'=r-1$; for any $t>r'$ and for any string
$x_{s:t}\in \X^{t-s+1}$, we define
\begin{align*}
&\tau_k\DEF \lfloor {t-k-s\over r'} \rfloor, \quad \mbox{and} \quad
\kappa_k(x_{s:t})\DEF \sum_{u=0}^{\tau_k-1}\mathbb{I}_{E}(x_{ur'+s+k:(u+1)r'+s+k}),\quad k=0,\ldots r'-1
\end{align*}
Thus $\kappa_0(x_{s:t})$ counts the number of vectors from set $E$
in the string $x_{s:t}$ in almost non-overlapping positions starting
from $s$. Here "almost non-overlapping positions" means that the
last entry of previous position and the first entry of the next one
overlap. Similarly, $\kappa_k(x_{s:t})$ counts the number of vectors
from set $E$ in the string $x_{s+k:t}$ $(k=0,\ldots, r'-1)$.
\\\\
 Let us also define reversed time counterpart of $\kappa_0$ as
follows
\begin{equation}\label{kappa'}
\bar{\kappa}(x_{s:t})\DEF \sum_{u=0}^{\tau_0-1}\mathbb{I}_{E}(x_{t-(u+1)r':
t-ur'}).
\end{equation}
Thus also $\bar{\kappa}(x_{s:t})$ counts the number of vectors from
set $E$ in the string $x_{s:t}$ in almost non-overlapping positions;
the difference with $\kappa$ is that  $\bar{\kappa}$ starts counting
from $t$. Note that  with $k=(t-s) \mod r'$
$\bar{\kappa}(x_{s:t})=\kappa_k(x_{s:t})$.
\begin{proposition}\label{prop:TVbound} Suppose \textbf{A1} is satisfied, $n \geq t \geq s \geq l \geq 1$ and $m \geq 1$.
Then with $\rho=1-n_0^{-2}$,  the following inequality holds for
every  $k=0,\ldots,r'-1$,
\begin{equation}\label{kpa}
\| \nu^{t}_{l:n;m}-\nu^{t}_{s:n;m}\|_{\rm TV} \leq
2\rho^{\kappa_k(X_{s:t})}, \quad \mbox{a.s.}
\end{equation}
Moreover, the inequality (\ref{kpa}) also holds when
$\kappa_k(X_{s:t})$ is replaced be  $\bar{\kappa}(X_{s:t}).$
\end{proposition}
\begin{proof} By (\ref{decomp}),
\begin{align*}
\nu^t_{l:n;m}=\nu^{s}_{l:n;1}F_{t-s;m}[X_{s:n}],\quad \nu^t_{s:n;m}=\nu^s_{s:n;1}F_{t-s;m}[X_{s:n}].
\end{align*}
Note that
$$F_{t-s;m}[X_{s:n}]=F_{k;1}[X_{s:n}]
U[X_{s+k:n}]U[X_{s+k+r':n}] \cdots
U[X_{s+k+(\tau_k-1)r':n}]F_{t-s-k-\tau_k r';m}[X_{s+k+\tau_k r':n}].$$
Thus
\begin{align*}
&\| \nu^t_{l:n;m}-\nu^t_{s:n;m}\|_{\rm TV}\\
&=\|(\nu^s_{l:n;1}-\nu^s_{s:n;1})F_{k;1}[X_{s:n}]U[X_{s+k:n}]U[X_{s+k+r':n}] \cdots
U[X_{s+k+(\tau_k-1)r':n}]F_{t-s-k-\tau_k r';m}[X_{s+k+\tau_k r':n}] \|_{\rm TV} \\
&\leq  2 \delta(U[X_{s+k:n}])\delta(U[X_{s+k+r':n}]) \cdots
\delta(U[X_{s+k+(\tau_k-1)r':n}]),
\end{align*}
where $\delta(U)$ denotes the Dobrushin coefficient of matrix $U$.
Note that if  $x_{1:n}$ is such that  $p(x_{1:n})>0$, then for any
$u=1,\ldots,n-1 $ there exists a state $y_u$ such that
$p(x_{u+1:n}|x_u,y_u)>0$ so that the assumption of Lemma
 \ref{lem:Doeblin} is fulfilled. Since $p(X_{1:n})>0$, a.s., we have
 by  Lemma
 \ref{lem:Doeblin}
 $$\delta(U[X_{s+k+ur':n}]) \leq
\rho \mathbb{I}_{E}(X_{s+k+ur':s+k+(u+1)r'})\quad u=0,\ldots, \tau_k-1,$$
and
so the statement follows.\\\\
Since for some $k$, $\bar{\kappa}[x_{s:t}]=\kappa_k[x_{s:t}]$, we
have that $\max_k \kappa_k[X_{s:n}]\geq \bar{\kappa}[X_{s:t}]$ and
so the second statement follows.
\end{proof}
\\\\
We are now ready to prove the first of the two main results of the
paper. Recall that we do not assume any specific initial distribution $\pi$ of the chain $Z$, hence all a.s.- statements below are with respect to the measure ${\bf P}$ in underlying probabilty space.
\begin{theorem} \label{th:forgetting} Assume \textbf{A1}-\textbf{A2} and let  $Z$ be Harris recurrent.
 \begin{enumerate}[label=(\roman*)]
\item Then for all $l, s\geq 1$
\begin{equation*}
\sup_{n \geq t}\sup_{m\geq 1}\| \nu^t_{l:n;m}-\nu^t_{s:n;m}\|_{\rm
TV} \xrightarrow[t]{} 0, \quad \mbox{a.s.}
\end{equation*}
\item If $Z$ is positive Harris, then there exists a constant $1>\alpha>0$ such that the following holds: for every $s \geq 1$ there exist a
$\sigma(X_{s:\infty})$-measurable
 random variable $C_s< \infty$ such that for all  $t \geq s \geq l \geq 1$
\begin{equation}\label{bound1}
\sup_{n\geq t}\sup_{m\geq 1}\| \nu^t_{l:n;m}-\nu^t_{s:n;m} \|_{\rm
TV} \leq C_s \alpha^{t-s}, \quad \mbox{a.s.}
\end{equation}
\end{enumerate}
\end{theorem}
\begin{proof}
\textbf{(i)} First we show that
\begin{align} \label{ineq:Pos}
P_z(X_{1:r} \in E)>0, \quad \forall z \in E_{(1)} \times \Y^+_{(1)}.
\end{align}
Recall that we denoted $E(x_1) = \{x_{2:r} \: | \: x_{1:r} \in E\}$.
We have for any $(x_1,i) \in E_{(1)} \times \Y^+_{(1)}$
\begin{align*}
P_{(x_1,i)}(X_{1:r} \in E)&= \int_{E(x_1)} \sum_{j \in \Y} p_{ij}(x_{1:r}) \, \mu^{r-1}(dx_{2:r})\geq  \int_{E(x_1)} \sum_{j \in \Y^+_{(2)}} \dfrac{1}{n_0}\, \mu^{r-1}(dx_{2:r})
=\dfrac{|\Y^+_{(2)}|}{n_0} \mu^{r-1}(E(x_1))>0,
\end{align*}
and so \eqref{ineq:Pos} holds. Here the first inequality follows
from \textbf{A1} and \eqref{ineq:n0Bound}, and the second inequality
follows from \textbf{A2}. Since by \textbf{A2} $\psi(E_{(1)} \times
\Y^+_{(1)})>0$, then it follows from Lemma \ref{lem:io} and
\eqref{ineq:Pos} that $X$ goes through $E$ infinitely often a.s.
Assuming for the sake of concreteness that $s \geq l$, we have that
there must exist $T(s) \in \{1, \ldots,r'\}$ such that
$\kappa_0(X_{s+T:s+T+u}) \xrightarrow[u]{} \infty$, a.s.. Thus, as
$u\to \infty$, we have $\max_{k\in \{0,\ldots,r'-1\}}
\kappa_k(X_{s:u})\to \infty$ and so the first part of the statement
follows from Proposition \ref{prop:TVbound}.
\\\\
\textbf{(ii)} Define $\ZZ_{k}=Z_{k:k+r-1}$, $k \geq 1$. From
Proposition \ref{block} we know that $\ZZ$ is a positive Harris
chain  with maximal irreducibility measure $\psi_r$. Recall that the
chain $\ZZ$ admits a cyclic decomposition $\{D_k, k=0,\ldots,
d-1\}$, where $d$ denotes the period of $\ZZ$. Also recall that by
\textbf{A2} $\psi(E_{(1)} \times \Y^+_{(1)})>0$; hence by (\ref{moot}) and
\eqref{ineq:Pos} $\pp(E \times \Y^r)>0$. Thus with no loss of
generality we may assume that
\begin{align} \label{ineq:Epp}
\pp\Big(D_0 \cap (E \times \Y^r)\Big)>0.
\end{align}
Let $s \geq 1$, and let $T(s) \geq 0$ be a
$\sigma(X_{s:\infty})$-measurable integer-valued random variable
defined as the smallest integer such that $\ZZ_{s+T} \in D_0$. Since
$\ZZ$ is Harris recurrent, then $T<\infty$, a.s.. We have thus by
the strong Markov property that $\{\ZZ_{s+T+k}\}_{k \geq 0}$ is a
Markov chain with the same transition kernel as $\ZZ$, hence also
positive Harris. Also, by the cyclic decomposition of $\ZZ$ and by
the fact that $\ZZ_{s+T} \in D_0$, we have that the Markovian
sub-process $\{\ZZ_{s+T+kd}\}_{k \geq 0}$ can be seen seen as the
{\it process $\ZZ$ on $D_0$}, i.e. as a process that starts from
$\ZZ_{s+T}$, the next value is the one of $\ZZ$ at the next visit of
$D_0$ and so on. With this observation, it is easy to see that
$\{\ZZ_{s+T+kd}\}_{k \geq 0}$ is $\pp\mid_{D_0}$-irreducible
($\pp\mid_{D_0}$ stands for restriction), positive Harris (if
$\pp\mid_{D_0}(A)>0$, then also $\pp(A)>0$ and so for every $z_1\in
D_0$, $P(\ZZ_k\in A, \quad {\rm i.o}\mid Z_1=z_1)=1$, so it is
Harris; since the restriction of invariant probability measure of
$\ZZ$ to $D_0$ is the invariant measure of $\ZZ$ on $D_0$,
\cite[Th.10.4.9]{MT}, we see that the process on $D_0$ has a
positive invariant measure) and aperiodic (aperiodicity here follows
from the fact that $d$ is defined as the largest cycle length
possible). It then follows  \cite[Th. 9.1.6]{MT} that the Markov
chain $\{\ZZ_{s+T+k d r'}\}_{k \geq 0}$ is positive Harris, having
the same invariant probability measure as the process
$\{\ZZ_{s+T+kd}\}_{k \geq 0}$. This invariant probability measure is
the one of $\ZZ$ conditioned on the set $D_0$, hence $P_{\pi}(\ZZ_1
\in \cdot\mid \ZZ_1\in D_0)$, where $\pi$ denotes the invariant
distribution of $Z$. Define
\begin{align*}
S(n)\DEF\sum_{k=0}^{n-1} \mathbb{I}_{E \times \Y^r}(\ZZ_{s+T+k d r'})
\quad \mbox{and} \quad p_0\DEF P_\pi(\ZZ_1 \in (E \times
\Y^r)\mid \ZZ_1\in D_0)={P_\pi(\ZZ_1 \in  D_0 \cap (E \times
\Y^r))\over P_{\pi}(\ZZ_1\in D_0)}.
\end{align*}
Since the invariant measure $P_{\pi}(\ZZ_1 \in \cdot)$ dominates the
maximal irreducibility measure $\pp$ \cite[Prop. 10.1.2(ii)]{MT},
then by \eqref{ineq:Epp}, $P_\pi(\ZZ_1 \in  D_0 \cap (E \times
\Y^r))>0$ and so $p_0>0$. By SLLN for positive Harris chains
\cite[Th. 17.1.7]{MT}
\begin{align} \label{conv:SLLN}
\lim_n\dfrac{1}{n}S(n)=p_0, \quad \mbox{a.s.}
\end{align}
We have for all $u \geq 0$
\begin{align} \label{ineq:kap}
\kappa_0(X_{s+T:s+T+u}) \geq S(\lfloor u/(dr') \rfloor).
\end{align}
To see (\ref{ineq:kap}) note that by the definition of $\ZZ$,
$$\ZZ_{s+T+(n-1)dr'}=Z_{s+T+(n-1)dr':s+T+((n-1)d+1)r'}.$$ Thus, when
$n\leq {u\over d r'}$, then $(n-1)d\leq (u/r'-1)$ and
$$s+T+\big((n-1)d+1\big)r'\leq s+T+u.$$
By (\ref{conv:SLLN}), for every $0<p<{p_0\over dr'}$, there exists a
$\sigma(X_{s:\infty})$-measurable finite random variable $U$
(depending on $s$ and $p$) such that for all $k \geq 0$,
\begin{align*}
S\Big(\lfloor {U+k\over dr'} \rfloor \Big)> (p dr'){(U+k)\over
dr'}\quad \Rightarrow \quad \kappa_0(X_{s+T:s+T+U+k})>(U+k)p.
\end{align*}
Therefore, if $t\geq s+T+U$, by taking $k=t-(s+T+U)$, we have
$$\kappa_0(X_{s:t})\geq
\kappa_0(X_{s+T:t})=\kappa_0(X_{s+T:s+T+U+k})\geq (U+k)p=p(t-s-T).$$
If $t<s+T+U$, then $T+U>t-s$ and defining $\alpha \DEF \rho^p$ and
$N\DEF T+U$, we have that by Proposition \ref{prop:TVbound} for any
$t \in \{s, \ldots,n\}$ the following inequalities hold  a.s.
\begin{align*}
\| \nu^t_{l:n;m}-\nu^t_{s:n;m} \|_{\rm TV}& \leq
2\rho^{\kappa_0(X_{s:t})}\leq 2\rho^{p(t-s-T)\mathbb{I}(N \leq t-s)}
\leq 2\alpha^{-N} \cdot \alpha^{t-s}.
\end{align*} So the statement holds with $C_s\DEF 2\alpha^{-2N}$.  \end{proof}
\begin{corollary}\label{korra}
 Assume \textbf{A1}-\textbf{A2} and let  $Z$ be Harris recurrent.
\begin{enumerate}[label=(\roman*)]
\item Then for all $l, s \geq 1$
\begin{equation}\label{cor1}
\lim_t\sup_{n\geq t}\|P(Y_{t:\infty}\in
\cdot|X_{l:n})-P(Y_{t:\infty}\in \cdot|X_{s:n})\|_{\rm TV}=0,\quad
\text{a.s.}\end{equation}
\item If $Z$ is positive Harris, then there exists a constant $1>\alpha>0$ such that the following holds: for every $s \geq 1$ there exist a
$\sigma(X_{s:\infty})$-measurable
 random variable $C_s< \infty$ such that for all $t\geq s \geq l \geq 1$
\begin{equation}\label{bound11}
\sup_{n\geq t}\|P(Y_{t:\infty}\in \cdot|X_{l:n})-P(Y_{t:\infty}\in
\cdot|X_{s:n}) \|_{\rm TV} \leq C_s \alpha^{t-s}, \quad \mbox{a.s.}
\end{equation}
\end{enumerate}
\end{corollary}
\begin{proof} Let ${\cal A}$ be the  algebra consisting of all cylinders of $\Y^{\infty}$. Thus ${\cal F}=\sigma({\cal A})$.   The statement (i) of Theorem \ref{th:forgetting}
means that for ${\bf P}$-a.s.,
\begin{equation*}\label{y}
\sup_{n\geq t}\sup_{A\in {\cal
A}}|\nu^t_{l:n;\infty}(A)-\nu^t_{s:n;\infty}(A)|\to
0.\end{equation*} Since for every two probability measures $P$ and
$Q$ on ${\cal F}=\sigma({\cal A})$, it holds by monotone class theorem that $$\sup_{A\in {\cal
A}}|P(A)-Q(A)|=\sup_{F\in \sigma({\cal A})}|P(F)-Q(F)|,$$ we have as
$t\to \infty$, $$\sup_{n\geq t}\|P(P_{t:\infty}\in
\cdot|X_{l:n})-P(Y_{t:\infty}\in \cdot|X_{s:n})\|_{\rm
TV}=\sup_{n\geq t}\sup_{F\in {\cal
F}}|\nu^t_{l:n;\infty}(F)-\nu^t_{s:n;\infty}(F)|\to 0,\quad {\rm
a.s.}$$ The proof of the second statement is the same.
\end{proof}
\\\\
By Levy martingale convergence theorem, for every $l$, $t$ and $F\in
{\cal F}$
$$\lim_n P(Y_{t:\infty}\in F
|X_{l:n})=P(Y_{t:\infty}\in F|X_{l:\infty}),\quad {\rm a.s.}.$$
Since ${\cal F}$ is countable generated, then  (\ref{cor1}) implies by Dynkin $\pi-\lambda$ theorem
\begin{equation}\label{mutter}
\lim_t\|P(Y_{t:\infty}\in \cdot|X_{l:\infty})-P(Y_{t:\infty}\in
\cdot|X_{s:\infty})\|_{\rm TV}=0.\end{equation}
Similarly, from (\ref{bound1}), it follows ($s\geq l$)
\begin{equation}\label{vater}
\|P(Y_{t:\infty}\in \cdot|X_{l:\infty})-P(Y_{t:\infty}\in
\cdot|X_{s:\infty}) \|_{\rm TV} \leq C_s \alpha^{t-s}, \quad
\mbox{a.s.}
\end{equation}

{\bf Remark.} Suppose we have two different initial distributions,
say $\pi$ and $\tilde \pi$ of $Z_1$, where $\tilde{\pi}\succ\pi $
(to avoid zero-likelihood observations). Let $\nu^t_{s:n;m}$ and
$\tilde{\nu}^t_{s:n;m}$ ($1\leq s\leq t\leq n$) be the corresponding
smoothing distributions. It is easy to see that the very proof of
Proposition \ref{prop:TVbound} yields for every $k=0,\ldots r'-1$,
\begin{equation}\label{tilde}
\|\nu_{s:n;m}^t-\tilde{\nu}^t_{s:n;m}\|_{\rm TV}\leq
2\rho^{\kappa_k(X_{s:t})},\quad P_{\pi}-{\rm  a.s.}
\end{equation}
Therefore also the statements of Corollary \ref{korra} hold: under
{\bf A1} and {\bf A2}, if $Z$ is Harris recurrent, then as $t\to
\infty$,
\begin{equation}\label{tilde1}
\sup_{n\geq t}\|P_{\pi}(Y_{t:\infty}\in
\cdot|X_{s:n})-P_{\tilde{\pi}}(Y_{t:\infty}\in \cdot|X_{s:n})\|_{\rm
TV}\to 0,\quad  P_{\pi}-{\rm a.s.}.
\end{equation}
As mentioned in Introduction, for $n=t$, such convergences -- {\it
filter stability} -- are studied by van Handel {\it et al.} in
series of papers \cite{TvH12,vH09a,vH09b,vH08,vH12}. If $Z$ is
positive Harris, then there the convergence above holds in
exponential rate, i.e. there exists an almost surely finite random
variable $C_s$ and $\alpha\in(0,1)$ so that
\begin{equation}\label{tilde2}
\sup_{n\geq t}\|P_{\pi}(Y_{t:\infty}\in
\cdot|X_{s:n})-P_{\tilde{\pi}}(Y_{t:\infty}\in \cdot|X_{s:n})\|_{\rm
TV} \leq K_s \alpha^{t},\quad  P_{\pi}-{\rm a.s.},
\end{equation}
where $K_s=C_s\alpha^{-s}$. Of course,  just like in (\ref{mutter}) and (\ref{vater}), we have
that (\ref{tilde1}) and (\ref{tilde2}) also hold with $n=\infty$.
The  convergence (\ref{tilde1}) with $n=\infty$ is studied by
van Handel in \cite{vH09b} (mostly in HMM setting) under the name
{\it weak ergodicity of Markov chain in random environment.}
\subsection{Two-sided forgetting}
\paragraph{A1 and A2 under stationarity.}
In this section we consider a two-sided stationary extension of $Z$,
namely $\{Z_k\}_{\in \mathbb{Z}}=\{(X_k,Y_k)\}_{k \in \mathbb{Z}}$. As previously, the process is defined on underlying probability space $(\Omega,{\cal F},{\bf P})$, but in this section,
the measure ${\bf P}$ is such that the process $Z$ is stationary. All a.s. statements are with respect to ${\bf P}$.
Denote for $n \geq 1$ and $x_{1:n}\in \X^n$,
\begin{align*}
\Y^*(x_{1:n})\DEF \{(y_{1},y_n) \:|\exists y_{2:n-1}: \: p(x_{1:n},y_{1:n})>0\}.
\end{align*}
In the stationary case it is convenient to replace \textbf{A1} and
\textbf{A2} with the following conditions.
\begin{description} \item[A1'] There exists a set $E \subset \X^r$, $r>1$, such that $\Y^* \DEF \Y^*(x_{1:r}) \neq \emptyset$ is the same for any $x_{1:r} \in E$, and $\Y^*=\Y^*_{(1)} \times \Y^*_{(2)}$.
\item[A2'] It holds $\mu^r(E)>0$.
\end{description}
It is easy to see that $\bf{A1'}$ and $\bf{A2'}$ imply
\begin{align} \label{ineq:Epos}
P(X_{1:r} \in E)>0.
\end{align}
Let us compare conditions {\bf A1} and ${\bf A1'}$. Suppose $E$ is
any set such that $\Y^*(x_{1:r})$ is the same for any $x_{1:r}\in E$
and also $\Y^+(x_{1:r})$ is the same for any $x_{1:r}\in E$. Then
clearly $\Y^*_{(1)}\subset \Y^+_{(1)}$, and these sets are equal, if
for every $i\in \Y^+_{(1)}$, there exists $x_1\in E_{(1)}$ such that
$p(x_1,i)>0$ (recall that we consider stationary case, thus
$p(x_1,i)=p(x_t,y_t=i)$ for any $t$). Therefore, if there exists
$i\in \Y^+_{(1)}\setminus \Y^*_{(1)}$, then $p(x_1,i)=0$ for every
$x_1\in E_{(1)}$. This implies that such a state $i$ almost never
occurs  together with an observation $x_1$ from $E_{(1)}$ and
without loss of generality we can leave such states out of
consideration. Indeed, recall the proof of Proposition
\ref{prop:TVbound}, where for given $x_{s:n}$, we calculated, for
any $k=0,\ldots, r'-1$,
$$
\nu_{s:n;m}^t=\nu^s_{s:n;1}[x_{s:n}]F_{k;1}[x_{s:n}]U[x_{s+k:n}]U[x_{s+k+r':n}]
\cdots U[x_{s+k+(\tau_k-1)r':n}]F_{t-s-k-\tau_k r';m}[x_{s+k+\tau_k
r':n}]
$$
so that for any $v\in \Y^m$ and for any $k$
\begin{align*}
&\nu_{s:n;m}^t(v)=\sum_{i_0,i_1,i_2\cdots i_{\tau_k}}
p(y_{s+k}=i_0|x_{s:n})p(y_{s+k+r'}=i_1|y_{s+k}=i_0,x_{s+k:n})p(y_{s+k+2r'}=i_2|y_{s+k+r'}=i_1,x_{s+k+r':n})\cdots\\
&\cdots p(y_{s+k+\tau_k
r'}=i_{\tau_k}|y_{s+k+(\tau_k-1)r'}=i_{\tau_k-1},x_{s+k+(\tau_k-1)r':n})p(y_{t:t+m-1}=v|y_{s+k+\tau_k
r'}=i_{\tau_k}, x_{s+k+{\tau}_k r':n}).
\end{align*}
Now observe: when $p(x_{s:n})>0$, but $p(x_{s+k},y_{s+k}=i_0)=0$,
then also $p(y_{s+k}=i_0\mid x_{s:n})=0$, and such $i_0$ could be
left out from summation. Similarly, if, for a $l=1,\ldots,\tau_k-1$,
$p(x_{s+k+lr'},y_{s+k+lr'}=i_l)=0$, we have that
$p(y_{s+k+lr'}=i_l|y_{s+k+(l-1)r'}=i_{l-1},x_{s+k+(l-1)r':n})=0$ and
such $i_l$ can left out from summation. Therefore, in what follows,
without loss of generality, we shall assume $\Y^+_{(1)}=\Y^*_{(1)}$.
As the following proposition shows, in this case the conditions
${\bf A1,A2}$ and ${\bf A1',A2'}$ are equivalent.
\begin{proposition} Let $Z$ be stationary. Then ${\bf
A1,A2}$ implies ${\bf A1', A2'}$. If ${\bf A1', A2'}$ holds and the
corresponding set $E$ is such that $\Y^*_{(1)}=\Y^+_{(1)}$, then
${\bf A1,A2}$ and ${\bf A1',A2'}$ are equivalent.\end{proposition}
\begin{proof} Assume ${\bf
A1}$ and ${\bf A2}$ hold and let $E$ be the corresponding set. Since
$\pi$ is stationary probability measure, it is equivalent to $\psi$
\cite[Thm. 10.4.9]{MT}. Therefore, by {\bf A2}, $\pi(E_{(1)}\times
\Y^+_{(1)})>0$ and so there exists $i^*\in \Y^+_{(1)}$ such that
$\pi(E_{(1)}\times \{i^*\})>0$. Consequently, there exists a set
$U\subset E_{(1)}$ such that $\mu(U)>0$ and $p(x,i^*)>0$ for every
$x\in U$. Let for every $x\in U$, $C(x)=\{j\in \Y: p(x,j)>0\}.$
Clearly there exists $U_o\subset U$ so that $\mu(U_o)>0$ and
$C(x)=C$ for every $x\in U_o$. Note that $i^*\in C\cap \Y^+_{(1)}$.
Define $E_o=\cup_{x\in U_o}\{x\}\times E(x)$. Since $\mu(U_o)>0$, we
have $\mu^r(E_o)>0$. For any $i\in C\cap \Y^+_{(1)}$, for any $j\in
\Y^+_{(2)}$ and for any $x_{1:r}\in E_o$, it holds that
$p(x_{1:r},y_1=i,y_r=j)=p_{ij}(x_{1:r})p(x_1,i)>0$ (because $i\in
\Y^+_{(1)}, j\in \Y^+_{(2)}$, $x_{1:r}\in E$ and so by {\bf A1}
$p_{ij}(x_{1:r})>0$; since $i\in C$ and $x_1\in U_o$, it also holds
that $p(x_1,i)>0$). Therefore, if $x_{1:r}\in E_o$, then
$\Y^*(x_{1:r})= \big(C\cap \Y^+_{(1)}\big)\times \Y^*_{(2)}$ and so
${\bf A1'}$ holds with $E_o$. Since $\mu^r(E_o)>0$, we have that
{\bf A2} holds as
well.\\
 Assume that ${\bf
A1'}$ and ${\bf A2'}$ hold and let $E$ be the corresponding set. If
$\Y^*_{(1)}=\Y^+_{(1)}$, then for every $x_1\in E_{(1)}$ and $i\in
\Y^+_{(1)}$, it holds $p(x_1;i)>0$, and as argued above, ${\bf A1'}$
implies ${\bf A1}$. By (\ref{ineq:Epos}), $\pi(E_{(1)}\times  \Y)>0$
and since $\pi$ is equivalent to $\psi$, it holds that
$\psi(E_{(1)}\times \Y)>0$. Thus ${\bf A1}$ and ${\bf A2}$ hold.
\end{proof}
\paragraph{Reversing the time.}
As previously, let for every $l \leq n$, $l,n \in \mathbb{Z} \cup
\{-\infty,\infty\}$, $t \in \mathbb{Z}$ and $m \geq 1$
$$\nu^t_{l:n;m}[x_{l:n}](v)\DEF P(Y_{t:t+m-1} =v| X_{l:n}=x_{l:n}),\quad v\in \Y^m$$
and, like before, $\nu^t_{l:n;m}$ denotes the random probability
distribution $\nu^t_{l:n;m-1}[X_{l:n}]$. Note that for any $z\in
\mathbb{Z}$, we have that the random vectors $\nu^t_{l:n;m}$ and
$\nu^{t+z}_{l+z:n+z;m}$ have the same distribution, thus w.l.o.g. we
shall consider the case $t=0$. Under ${\bf A1}$ and ${\bf A2}$, it
follows from (\ref{bound1}) using triangular inequality that for
every $-l_2\leq -l_1 \leq -s<0$ and $m\geq 1$, it holds
\begin{equation}\label{rev0}
\sup_{n\geq 0}\| \nu^0_{-l_1:n;m}-\nu^0_{-l_2:n;m} \|_{\rm TV}\leq
C_s\alpha^{s},\quad {\rm a.s.},
\end{equation}
where the random variable $C_s$ is
$\sigma(X_{-s:\infty})$-measurable and so depends on $s$. Assuming
that the reversed-time chain $\{\bar{Z}_{k}\}_{k\geq 0}$, where
$\bar{Z}_k\DEF Z_{-k}$ is also positive Harris satisfying {\bf A1}
and {\bf A2}, then (\ref{rev0}) implies
\begin{equation}\label{rev1}
\sup_{n\geq 0}\| \nu^0_{-n:l_3;m}-\nu^0_{-n:l_4;m} \|_{\rm TV}\leq
\bar{C}_s\bar{\alpha}^{s-m+1},\quad {\rm a.s.},
\end{equation}
where $1\leq m\leq s\leq l_4\leq l_3$, the random variable
$\bar{C}_s$ is $\sigma(X_{-\infty:s})$-measurable
 and $\bar{\alpha}\in (0,1)$.
\\\\
The following theorem shows that when using reversed-time chain and
backward-counter $\bar{\kappa}$, we have the inequality (\ref{rev0})
with
 $C_s$
replaced by another random variable $C_0$ that is
$\sigma(X_{-\infty:0})$-measurable, but independent of $s$.
Similarly, the random variable $\bar{C}_s$ could be replaced by a
$\sigma(X_{m-1:\infty})$-measurable random variable $\bar{C}_{m-1}$
that is also independent of $s$ (but dependent on $m$). Because of
the stationarity, the assumptions {\bf A1}-{\bf A2} are replaced by
the (formally) weaker assumptions ${\bf A1'}-{\bf A2'}$, but as we
argued, they can be considered to be equivalent. Throughout the
section we assume $m\geq 1$ is a fixed integer.
\begin{theorem}\label{thm2} Assume that $Z$ is a stationary and positive Harris chain such that  ${\bf A1'}$ and ${\bf A2'}$
hold. Assume also that  reversed-time chain  $\{\bar{Z}_{k}\}_{k\geq
0}$ is also Harris recurrent. Then there exists a
$\sigma(X_{-\infty:0})$-measurable random variable $C_0$ and
$\alpha\in (0,1)$ such that for every $-l\leq -s < 0<n$  it holds
\begin{equation}\label{rev}
\sup_{n\geq 0}\| \nu^0_{-l:n;m}-\nu^0_{-s:n;m} \|_{\rm TV}\leq
C_0\alpha^{s},\quad {\rm a.s.}.
\end{equation}
There also exists a $\sigma(X_{m-1:\infty})$-measurable random
variable $\bar{C}_{m-1}$ and $\bar{\alpha}\in (0,1)$ such that for
every $s$ and $l$ such that  $m\leq s\leq l$,
\begin{equation}\label{rev2}
\sup_{n\geq 0}\| \nu^0_{-n:l;m}-\nu^0_{-n:s;m} \|_{\rm TV}\leq
\bar{C}_{m-1}{\bar \alpha}^{s-(m-1)}\quad {\rm a.s.}.
\end{equation}
\end{theorem}
\begin{proof}
By stationarity the reversed-time chain ${\bar Z}$ is  positive
Harris. Now we apply the proof of \textbf{(ii)} of Theorem
\ref{th:forgetting} to the reversed-times block chain
$\bar{\ZZ}_{k}\DEF\bar{Z}_{k:k+r'}=(Z_{-k},\ldots, Z_{-k-r'})$, $k
\geq 0$. Let $f: \X^r\to \X^r$ be the mapping that reverses the
ordering of vector: $f(x_1,\ldots,x_r)=(x_r,\ldots,x_1)$, and let
$\bar{E}\DEF \{f(x_{1:r}):x_{1:r}\in E\}$. Thus $\bar{\ZZ}_{k}\in
\bar{E}\times \Y^r$ if and only if $Z_{-k-r':-k}\in E\times \Y^r$.
Now, just like in the proof  of \textbf{(ii)} of Theorem
\ref{th:forgetting}, we define
\begin{align*}
S(n)\DEF\sum_{k=0}^{n-1} \mathbb{I}_{\bar{E} \times \Y^r}({\bar
\ZZ}_{T+k d' r'}),
\end{align*}
where, as previously, $T$ is a random variable so that
$\bar{\ZZ}_T\in D_0$, where $D_0,\ldots,D_{d'-1}$ is a cyclic
decomposition of $\bar{\ZZ}_k$. The set $D_0$ is such that
$P\big(\ZZ_1\in D_0\cap (E\times \Y^r)\big)>0$, by
(\ref{ineq:Epos}), such a set $D_0$  exists. Therefore, $S(n)/n\to
p_o$, a.s., where
$$p_o\DEF P\big(\bar{\ZZ}_1\in D_0\cap (\bar{E}\times \Y^r)\big|\bar{\ZZ}_1\in
D_0)>0.$$ We denote $\bar{X}_t \DEF X_{-t}$. Thus, for any $u$,
$\bar{X}_{T:T+u}=(X_{-T},\ldots,X_{-T-u})$ and so for any
$k=0,1,2,\ldots,$ $\bar{X}_{T+k d'r':T+(k+1)d'r'}\in \bar{E}$,
equivalently, $\bar{\ZZ}_{T+k d'r'}\in \bar{E}\times \Y^r$   only if
$X_{-T-(k+1) d'r':T-kd'r'}\in E$. Hence the inequality
(\ref{ineq:kap}) now is
\begin{equation}\label{ineqkap2}
{\kappa}_0( \bar{X}_{T:T+u})=\bar{\kappa}(X_{-T-u:-T})\geq
S\big(\big\lfloor {u\over (d'r')}\big\rfloor\big),
\end{equation}
where $\bar{\kappa}$ is defined as in (\ref{kappa'}). Now everything
is the same as in the proof of \textbf{(ii)} of Theorem
\ref{th:forgetting}: for every for  $0<p<{p_0\over d'r'}$, there
exists a finite random variable $U$ (depending on  $p$) such that
for all $k \geq 0$,
\begin{align*}
 \bar{\kappa}(\bar{X}_{T:T+U+k})>(U+k)p.
\end{align*}
Therefore, if $s\geq T+U$, by taking $k=s-(T+U)$, we have
$$\bar{\kappa}(\bar{X}_{0:s})>(U+k)p.$$ By assumption all conditions of   Proposition
\ref{prop:TVbound} and applying it with $-l\leq -s\leq 0\leq n$ (and
with $\bar{\kappa}$), we obtain just like in the proof of
\textbf{(ii)} of Theorem \ref{th:forgetting}
\begin{align*}
\| \nu^0_{-l:n;m}-\nu^0_{-s:n;m} \|_{\rm TV}& \leq
2\rho^{\bar{\kappa}(X_{-s:0})}\leq 2\rho^{p(s-T)\mathbb{I}(T+U\leq
s)} \leq 2\alpha^{-(U+T)} \cdot \alpha^{s},\quad \text{a.s.}.
\end{align*}
So the statement holds with $C_0\DEF 2\alpha^{-(U+T)}$ and the the
random variable $C_0$ is independent of $s$. This proves
(\ref{rev}).
\\
If $Z$ satisfies ${\bf A1'}$ and ${\bf A2'}$, then the reversed-time
chain $\bar{Z}$ satisfies ${\bf A1'}$ and ${\bf A2'}$ with $\bar{E}$
instead of $E$. Then the inequality (\ref{rev}) applied to $\bar{Z}$
yields to (\ref{rev2}). The constants might be different, because
the transition kernel of $\bar{Z}$ might be different from that of
$Z$.\end{proof}
\paragraph{The limits.}
By Levy's martingale convergence theorem, for every $s$ there exists
limits (recall that $\nu^0_{-s:n;m}$ are just $\mid
\Y\mid^m$-dimensional random vectors)
\begin{equation}\label{lim1}
\lim_n \nu^0_{-s:n;m}\DEF \nu^0_{-s:\infty;m},\quad {\rm a.s.},\quad
\lim_l \nu^0_{-l:\infty;m}\DEF \nu^0_{-\infty:\infty;m},\quad {\rm
a.s.}.
\end{equation}
Plugging (\ref{lim1}) into (\ref{rev}), we obtain
\begin{equation}\label{rev-inf1}
\| \nu^0_{-\infty:\infty;m}-\nu^0_{-s:\infty;m} \|_{\rm TV}\leq
C_0\alpha^{s}\quad \rm{a.s.}.
\end{equation}
Similar, for any $s>1$, the limit
\begin{equation}\label{lim2}
\lim_l \nu^0_{-s:l;m}\DEF \nu^0_{-s:\infty;m},\quad \rm{a.s.}
\end{equation}
exists and plugging (\ref{lim2}) into
 (\ref{rev2}), we obtain for any $s'>m-1$
\begin{equation}\label{rev-inf2}
\| \nu^0_{-s:\infty;m}-\nu^0_{-s:s';m} \|_{\rm TV}\leq
\bar{C}_{m-1}\bar{\alpha}^{s'-(m-1)},\quad \rm{a.s.}.
\end{equation}
The inequalities (\ref{rev-inf1}) and (\ref{rev-inf2}) together
imply the following approximation inequality
\begin{equation}\label{sm}
\| \nu^0_{-\infty:\infty;m}-\nu^0_{-s:s';m} \|_{\rm TV}\leq
C_0\alpha^{s}+ \bar{C}_{m-1}\bar{\alpha}^{s'-(m-1)},\quad \rm{a.s.}.
\end{equation}
Applying (\ref{sm}) to $\nu^t_{1:n;m}$, we obtain the following
corollary.
\begin{corollary}\label{korra2} Suppose the assumptions of Theorem \ref{thm2} hold. Then there exists $\alpha_o\in (0,1)$
such that for every $n,t$ satisfying $n\geq t+m-1$, $t\geq 1$, it
holds
\begin{equation}\label{2poolt}
\|P(Y_{t:t+m-1}\in \cdot|X_{1:n})-P(Y_{t:t+m-1}\in
\cdot|X_{-\infty:\infty})\|_{\rm TV}\leq C_t \alpha_o^{(t-1)\wedge
(n-t-m+1)},\quad {\rm a.s.},
\end{equation}
where $C_t$ is a $\sigma(X_{-\infty:t},X_{t+m-1:\infty})$-measurable
random variable.
\end{corollary}
With inequalities (\ref{rev0}) (letting first $n\to \infty$ and then $l_2\to \infty$)  and (\ref{rev1}) (with $n=l_1$ and $l_3\to \infty$), the approximation
inequality (\ref{sm}) would be
\begin{equation}\label{sm1}
\| \nu^0_{-\infty:\infty;m}-\nu^0_{-l_1:l_4;m} \|_{\rm TV}\leq
C_s\alpha^{s}+ \bar{C}_{s'}\bar{\alpha}^{s'},\quad \rm{a.s.},
\end{equation}
where, $l_1\geq s$ and $l_4\geq s'$, the random variables $C_s$ and
$\bar{C}_{s'}$ depend on $s$ and $s'$, respectively. Applying this
inequality to $\nu^t_{1:n;m}$, we obtain the following counterpart
of Corollary \ref{korra2}
\begin{corollary}\label{korra3} Suppose the assumptions of Theorem \ref{thm2} hold.
Then there exist $\alpha,\bar{\alpha}\in (0,1)$ such that for every
$t,k,n$ satisfying $n\geq k\geq t+m-1$, $t\geq 1$ it holds
\begin{equation}\label{2poolt2}
\|P(Y_{t:t+m-1}\in \cdot|X_{1:n})-P(Y_{t:t+m-1}\in
\cdot|X_{-\infty:\infty})\|_{\rm TV}\leq C_1\alpha^{t-1}+\bar{C}_k\bar{\alpha}^{k-t-m+1} ,\quad {\rm a.s.},
\end{equation}
where , $C_1$ is $\sigma(X_{1:\infty})$-measurable and $\bar{C}_k$
is $\sigma(X_{-\infty:k})$-measurable.
\end{corollary}
Corollary \ref{korra3} is a PMM-generalization of Theorem 2.1 in
\cite{smoothing2} (see also \cite{kuljus}). As  mentioned in
the introduction, (\ref{2poolt2}) is very useful in many
applications of segmentation theory.
\paragraph{Ergodicity.} In (\ref{2poolt2}), we can replace 1 by any $l\in \mathbb{Z}$, and consider the stochastic process $\{C_l\}_{l\in \mathbb{Z}}$. The construction of $C_l$ reveals that for any $l$, $C_l=f(X_{l,\infty})$, where the function $f$ is independent of $l$. This means that the process $\{C_l\}_{l\in \mathbb{Z}}$ is a stationary coding of the process $Z$ (see e.g. \cite[Ex. I.1.9]{Shields} or \cite[Sec. 4.2.]{Gray}). Since stationary coding preserves stationarity and ergodicity (\cite[Lemma 4.2.3]{Gray} or \cite[Ex. I,2,12]{Shields}), we see that the process $\{C_l\}_{l\in \mathbb{Z}}$ is stationary (since $Z$ was assumed to be stationary) and, when $Z$ is ergodic process (in the sense of ergodic theory), then so is $\{C_l\}_{l\in \mathbb{Z}}$. The same holds for the process $\{\bar{C}_k\}_{k\in \mathbb{Z}}$. The ergodicity of these processes is key for proving the existence of limit $R$ in PMAP segmentation (recall paragraph "Applications in segmentation").
\section{Examples} \label{sec:Examples}
\subsection{Countable $\X$}
When $\X$ is countable, then $Z$ is a Markov chain with countable
state space and $Z$ is (positive) Harris recurrent if and only if
$Z$ is (positive) recurrent. If $\X$ is finite, then every
irreducible Markov chain is positive recurrent. If $\X$ is
countable, then  \textbf{A1} is fulfilled if and only if for some
$r>1$ there exists a vector $x_{1:r} \in \X^r$ such that
$\Y^+(x_{1:r})=\Y^+(x_{1:r})_{(1)}\times \Y^+(x_{1:r})_{(2)} \neq
\emptyset$. For irreducible $Z$, the assumption {\bf A2}
automatically holds if $\Y^+(x_{1:r})\ne \emptyset$ and that is
guaranteed by {\bf A1}. The interpretation of ${\bf A1'}$ in the
case of countable $\X$ is very straightforward: for every two
vectors $y_{1:r},\bar{y}_{1:r}\in \Y^r$ satisfying
$p(x_{1:r},y_{1:r})>0$ and $p(x_{1:r},\bar{y}_{1:r})>0$, there
exists a third vector $\tilde{y}_{1:r}\in \Y^r$ such that
 $\tilde{y}_1=y_1$,
$\tilde{y}_r=\bar{y}_r$ and $p(x_{1:r},\tilde{y}_{1:r})>0$. In
ergodic theory, this property is called as the {\it subpositivity}
of the word $x_{1:r}$ for {\it factor map} $\pi: {\cal Z}\to {\cal
X}, \pi(x,y)=x$, see (\cite{yoo}, {Def 3.1}). Thus ${\bf A1'}$
ensures that a.e. realization of $X$ process has infinitely many
subpositive words.
\subsection{Nondegenerate PMM's}\label{sec:nondeg} In \cite{TvH12,TvH14}, Tong and van Handel introduce the
non-degenerate PMM.  When adapted to our case, the model is {\it
non-degenerate} when the kernel density  factorizes as follows
\begin{equation}\label{non-deg}
q(x',j|x,i)=p_{ij}r(x'|x)g(x,i,x',j),
\end{equation}
where $\mathbb{P}=(p_{ij})$ is a transition matrix and $r(x'|x)$ is
a density of transition kernel, i.e for every $x$, $x'\mapsto
r(x'|x)$ is a density  with respect to $\mu$ so that $R(A|x)\DEF
\int_A r(x'|x)\mu(dx')$ is a transition kernel on $\X\times \B(\X)$.
The third factor $g(x,i,x',j)$ is a strictly positive measurable
function. For a motivation and general properties of non-degenerate
PMM's see \cite{TvH12}, the key property is that the function $g$ is
strictly positive. The non-degenerate property does not imply that
$Y$ is  a Markov chain and even if it is, its transition matrix need
not be $\mathbb{P}$. Under (\ref{non-deg}), for every $x_{1:n}$,
$n\geq 2$, $i,j\in \Y$
\begin{equation}\label{pos}
p_{ij}(x_{1:n})= p^{n-1}_{ij} g_n(i,j,x_{1:n})\prod_{k=1}^n
r(x_k|x_{k-1}),\end{equation} where $p^{n-1}_{ij}$ stands for the
$i,j$-element of  $\mathbb{P}^{n-1}$ and $g_n(i,j,x_{1:n})>0$, (see
also \cite[Lemma 3.1]{TvH12}). From (\ref{pos}) it immediately
follows that when $\mathbb{P}$ is primitive, i.e. for some $R\geq
1$, $\mathbb{P}^R$ has strictly positive entries, then any $x_{1:r}$
with $r=R+1$ such that $p(x_{1:r})>0$ satisfies {\bf A1}:
$\Y^+(x_{1:r})=\Y\times \Y$. Thus, when $\mathbb{P}$ is primitive,
then {\bf A1} and {\bf A2} both hold with $E=\{x_{1:r}:
p(x_{1:r})>0\}$.\\\\
Barely the non-degeneracy  is not sufficient for the primitivity of
$\mathbb{P}$. We now show that when combined with some natural
ergodicity assumptions, then $\mathbb{P}$ is primitive. Let
$P^n(i,j)\DEF P(Y_{n}=j|Y_1=i)$, $n>1$. Recall that $\pi$ is a
stationary measure of $Z$, and with a slight abuse of notation, let
$\pi(i)\DEF \pi(\{i\} \times \X)$ be a marginal measure of $\pi$.
Surely $\pi(i)>0$ for every $i\in \Y$ and so the convergence
\begin{equation}\label{ergo}
\sum_{i\in \Y}\pi(i)\|P^n(i,\cdot)-\pi(\cdot)\|_{TV}\to 0,
\end{equation}
equivalently, $P(Y_n=j|Y_1=i)\to \pi(j),\quad \forall i,j\in \Y$
implies that $P^n(i,j)$ must consist of all positive entries when
$n$ is big enough. If $Y$ happens to be a Markov chain with
transition matrix $\mathbb{P}$, then it is primitive. Otherwise
observe that by (\ref{pos})
$$P^n(i,j)=\int_{\X^{n}}p(x_1|y_1=i)p_{ij}(x_{1:n})\mu^{n}(dx_{1:n})=p^{n-1}_{ij}\int_{\X^{n}}p(x_1|y_1=i)g_n(i,j,x_{1:n})\prod_{k=1}^n
r(x_k|x_{k-1})\mu^{n}(dx_{1:n})$$ so that if there exists $n$ such
that $P^n(i,j)>0$ for every $i,j\in \Y$, then $\mathbb{P}^{n-1}$
consists of strictly positive entries and so it is primitive. Hence
for non-degenerate PMM's (\ref{ergo}) implies {\bf A1} and {\bf A2}.
A stronger version of  (\ref{ergo}) (so-called {\it marginal
ergodicity}) is assumed in \cite{TvH12} for proving the filter
stability for non-degenerate PMM's  \cite[Th 2.10]{TvH12}. Thus, for
finite $\Y$, Theorem \ref{th:forgetting} generalizes that result. We
believe that the  key assumption of non-negative $g$ can be relaxed
in the light of cluster-assumption introduced in the next subsection
for HMM's.
\subsection{Hidden Markov model}\label{HMM}
In case of HMM the transition kernel density factorizes as
$q(x',j|x,i)=p_{ij}f_j(x')$. Here $\mathbb{P}=(p_{ij})$ is the
transition matrix of the Markov chain $Y$ and $f_j$ are the {\it
emission densities} with respect to measure $\mu$. Thus
$$p_{ij}(x_{1:n})=\sum_{k_1,\ldots,k_{n-2}}p_{ik_1}f_{k_1}(x_2)p_{k_1 k_2}f_{k_2}(x_3)\cdots p_{k_{n-2}j}f_j(x_n).$$
Let  $G_i\DEF\{x
\:| \: f_i(x)>0\}$. The process $Z$ is irreducible (with respect to
some measure) if and only if $Y$ is irreducible and in this case the
maximal irreducible measure is
$$\psi\big(\cup_{i\in \Y} A_i\times \{i\}\big)=\mu\big(\cup_{i\in \Y} A_i\cap G_i\big),\quad  A_i\in {\cal
B}(\X).
$$
Since HMM's are by far the most popular PMM's in practice, it would
be desirable to have a relatively easy criterion to check the
assumptions {\bf A1} and {\bf A2} for HMM's. In this subsection, we
introduce  a fairly general but easily verifiable assumption called
{\it cluster assumption}. Lemma \ref{lem:HMMcond} below shows that
cluster assumption implies {\bf A1} and {\bf A2}. The rest of the
subsection is mostly devoted to show that the cluster assumption
still generalizes many similar assumptions encountered in the
literature.
\\\\
A subset $C\subset\Y$ is called a \emph{cluster}, if
\begin{equation}\label{weak-cl}
\mu \left[ \left(\cap _{i\in C}G_i \right)\setminus \left(\cup
_{i\notin C}G_i \right) \right]>0.
\end{equation}
 Surely, at least one cluster always exists. Also, it is important to observe that every state $i$ belongs to at least one cluster. Distinct clusters need not be disjoint and  a cluster can consist of
a single state.  \\\\
The {\bf cluster assumption} states:
 There exists a
cluster $C\subset\mathcal{Y}$ such that the sub-stochastic matrix
$\mathbb{P}_C=(p_{ij})_{i,j\in C}$ is { primitive}, that is
$\mathbb{P}^R_C$ has only positive elements for some positive
integer $R$.\\\\
Thus the cluster assumptions implies that the Markov cain $Y$ is
aperiodic but not vice versa -- for a counterexample consider a
classical example appearing in \cite{HMMbook} (Example 4.3.28) as
well as  in \cite{ChigFilter,vH15}. Let $\Y=\{0,1,2,3\}$,
$\X=\{0,1\}$ and let the Markov chain $Y$ be be defined by
$Y_k=Y_{k-1}+U_k \pmod 4$, where $\{U_k\}$ is an i.i.d. Bernoulli
sequence with $P(U_k=1)=p$ for some $p \in (0,1)$. The observations
are defined by $X_k=\mathbb{I}_{\{0,2\}}(Y_k)$ and the initial
distribution of $Y$ is given by $P(Y_1=0)=P(Y_1=1)=\frac{1}{2}$.
Here $G_0=G_2=\{1\}$ and $G_1=G_3=\{0\}$. Thus the clusters are
$\{0,2\}$ and $\{1,3\}$ and the corresponding matrices
$\mathbb{P}_C$ are both diagonal so  that cluster-condition is not
fulfilled. In this example also {\bf A1}-{\bf A2} is not fulfilled
-- indeed for every $x_{1:r} \in \X^r$, there exists a pair $i,j\in
\{0,1,2,3\}$, depending on $x_{1:r}$ so that
$\Y^+(x_{1:r})=\{(i,j),(i+2 \pmod 4,j+2 \pmod 4)\}$. Thus
$\Y^+(x_{1:r})_{(1)}=\{i,i+2 \pmod 4\}$ and
$\Y^+(x_{1:r})_{(2)}=\{j,j+2 \pmod 4\}$ and  $\Y^+(x_{1:r}) \neq
\Y^+(x_{1:r})_{(1)} \times \Y^+(x_{1:r})_{(2)}$. Finally, we observe
that forgetting properties fail. To see that observe:
 knowing $X_{1:n}$ one can completely determine the hidden sequence
$Y_{1:n}$. For example if $X_{1:8}=01110010$ then
$Y_{1:8}=12223301$. One the other hand from $X_{2:n}$ it is not
possible to fully determine any $Y_k$, provided $X_2 =1$. For
example, if $X_{2:8}=1110010$, then $Y_{1:8}$ is either 12223301
(when $X_1=0$) or 00001123 (when $X_1=1$). In particular we have for
$3 \leq t \leq n$,
\begin{align*}
\nu^t_{3:n;1}(i)&= \sum_{j \in \Y}P(Y_3=j|X_{3:n})P(Y_t=i|Y_{3}=j,X_{3:n})\\
&=X_3\sum_{j \in \{0,2\}}P(Y_3=j|Y_{3} \in \{0,2\})P(Y_t=i|Y_{3}=j,X_{3:n}) \\
& \quad +(1-X_3)\sum_{j \in \{1,3\}}P(Y_3=j|Y_{3} \in
\{1,3\})P(Y_t=i|Y_{3}=j,X_{3:n}).
\end{align*}
Now observe that there exists $1/2>\epsilon>0$
such that $\min_{j \in \{0,2\}}P(Y_3=j|Y_{3} \in \{0,2\})$ and
$\min_{j \in \{1,3\}}P(Y_3=j|Y_{3} \in \{1,3\})$ are both greater
than $\epsilon$ and, therefore, less than $1-\epsilon$. We therefore
obtain
\begin{align*}
&\nu^t_{3:n;1}(i)\geq \epsilon [X_3(\phi_t(i,0)+\phi_t(i,2))+(1-X_3)(\phi_t(i,1)+\phi_t(i,3))],\\
&\nu^t_{3:n;1}(i)\leq (1-\epsilon)
[X_3(\phi_t(i,0)+\phi_t(i,2))+(1-X_3)(\phi_t(i,1)+\phi_t(i,3))],
\end{align*}
where
 $\phi_{t}(i,j)=P(Y_t=i|Y_{3}=j,X_{3:n})$. Noting that $\phi_t(i,j)$ is always either zero or
 one and
 $$X_3(\phi_t(i,0)+\phi_t(i,2))+(1-X_3)(\phi_t(i,1)+\phi_t(i,3))=1,$$
 it follows that for every $i$,  $\nu^t_{3:n;1}(i)\in (\epsilon,1-2\epsilon)$.
 Since  for every $i$, $\nu^t_{1:n;1}(i)\in \{0,1\}$, we have that
\begin{equation*}
\| \nu^t_{1:n;1}-\nu^t_{3:n;1}\|_{\rm TV} \geq 2\epsilon
\end{equation*}
for all $3 \leq t \leq n$, and so neither (i) nor (ii) of Theorem
\ref{th:forgetting} holds.
\begin{lemma}\label{lem:HMMcond} Let $Z$ be hidden Markov chain with irreducible hidden chain $Y$. Then the cluster-assumption  implies \textbf{A1}-\textbf{A2}.
\end{lemma}
\begin{proof}
There must exist integer $R \geq 1$ such that $\mathbb{P}_C^R$
consists of only positive elements. Defining $\Y_C=\{i \in \Y \: |
\: p_{ij}>0, j \in C \}$ and taking
$$E=\left( \cup_{i \in \Y_C} G_i \right) \times \left[ \left(\cap _{i\in C}G_i \right)\setminus \left(\cup _{i\notin C}G_i \right) \right]^{R+1},$$
we have that \textbf{A1} holds with $\Y^+=\Y_C \times C$. Observe
that $E_{(1)}= \cup_{i \in \Y_C} G_i $. Since $$\psi(E_{(1)}\times
\Y^+_{(1)})=\mu\big(\cup_{i \in \Y_C} G_i\big)>0$$ we see that
\textbf{A2} also holds.
\end{proof}
\\\\
The cluster-assumption was introduced in \cite{IEEE,AV,VT2} in other
purposes than exponential forgetting. Later it was successfully
exploited in many different setups \cite{kuljus,peep,smoothing}. In
those earlier papers, the concept of cluster was stronger than
(\ref{weak-cl}), namely
 $C\subset\Y$ was called a cluster if
\begin{equation} \label{HMM-clustermu}
\mu\left(\cap _{i\in C}G_i\right )>0\quad {\rm and}\quad \mu \left[
\left (\cap _{i\in C}G_i\right) \cap \left (\cup _{i\notin C}G_i
\right) \right]=0.
\end{equation}
The weaker definition of cluster  (\ref{weak-cl}) was first
introduced in \cite{PMMinf}.
\\\\
We shall now show how in the case of finite
${\cal Y}$, the cluster-assumption naturally generalizes
many existing mixing conditions encountered in the literature. The
following assumption  is known as {\it strong mixing condition}
(Assumption 4.3.21 in \cite{HMMbook}): for every $x\in \X$, there
exists probability measure  $K_x$ on $\Y$ and strictly positive
functions $\zeta^-$, $\zeta^+$ on $\X$ such that
\begin{equation}\label{strHMM}
\zeta^-(x)K_x(j)\leq p_{ij}f_j(x)\leq \zeta^+(x)K_x(j)\quad \forall
i\in \Y.
\end{equation}
A stronger version of the strong mixing condition is the following:
there exists positive numbers $\sigma^-$ and $\sigma^+$ and a
probability measure $K$ on $\Y$ such that
\begin{equation}\label{strHMM2}
\sigma^-K(j)\leq p_{ij}\leq \sigma^+K(j),\quad \forall i\text{   and
}\quad 0<\sum_jK(j)f_j(x)<\infty,\quad \forall x\in \X.
\end{equation}
This is Assumption 4.3.24 in \cite{HMMbook}. It is easy to verify
that under the strong mixing condition the Dobrushin coefficent of
$r$-step transition matrix $U(x_{s:n})=F_{r-1;1}(x_{s:n})$ can be
bounded above by $$\delta(U)\leq
\prod_{i=s+1}^{s+r-1}\Big(1-{\zeta^{-}(x_i)\over
\zeta^{+}(x_i)}\Big).$$ Under (\ref{strHMM2}) the upper bound
$(1-{\sigma^+\over \sigma^-})^{r-1}$ --  a constant less than 1. Now
it is clear that under (\ref{strHMM2}) the exponential forgetting
holds with
non-random universal constant $C^*$, i.e. in the inequality (\ref{bound1})  $C_s\equiv C^*$ for every $s$.\\\\
In the book \cite{HMMbook}, the Assumptions 4.3.21 and 4.3.24 as well as Assumptions 4.3.29 and 4.3.31 below are
stated for general  state space model, where ${\cal Y}$ is general space, and so (\ref{strHMM}) and (\ref{strHMM2}) are just the versions  of these assumptions for the discrete (finite or countable infinite) ${\cal Y}$. We now briefly argue that for the case of discrete  ${\cal Y}$ they are rather restrictive and our cluster-assumption naturally generalizes them. Indeed, it  is easy to see that
(\ref{strHMM2}) holds if $p_{ij}>0$ for every $i,j$ and for every
$x$, there exists $j$ so that $f_j(x)>0$ (this is a very natural
condition, otherwise leave $x$ out of $\X$). On the other hand, if
the transition matrix is is irreducible then every row has at least
one positive entry and then (\ref{strHMM2}) implies that $p_{ij}>0$
for every $i,j$ -- a rather strong restriction on transition matrix.
The same holds for (\ref{strHMM}). Indeed, since for every $j$,
there exists $x$ so that $f_j(x)>0$   and for every $j$ there exists
$i$ such that $p_{ij}>0$ (implied by irreducibility), then for every
$j$ there exists $x$ and $i$ so that $ p_{ij}f_j(x)>0$. Then
$p_{i'j}>0$ for every $i'$ so that all entries of transition matrix
are positive. If the entries of $\mathbb{P}$ are all positive (as it
is sometimes assumed, e.g. \cite{forgInit2}), then any cluster
satisfies the requirement of cluster assumption (with $R=1$), so
that strong mixing condition trivially implies cluster-assumption.
\\\\
In order to incorporate zero-transition, the primitivity of one-step
transition matrix $\mathbb{P}$ could be replaced by that of $R$-step
transition matrix for some $R>1$. An example of such kind of mixing
assumptions is the following (Assumption 4.3.29 in \cite{HMMbook},
see also \cite{Ung3,zeitouni}): There exists positive numbers
$\sigma^-$ and $\sigma^+$, an integer $R$ and a probability measure
$K$ on $\Y$ such
 that with $p^R_{ij}$ being $i,j$-element of $\mathbb{P}^R$, we have
\begin{enumerate}
\item $\sigma^-K(j)\leq p^R_{ij}\leq \sigma^+K(j),\quad \forall i,j;$
\item $f^-(x)\leq \min_i f_i(x)\leq \max_i f_i(x)\leq f^+(x)\quad\forall i\quad
\text{  where  } f^-,f^+ \text{are strictly positive functions.}$
\end{enumerate}
When the densities are bounded away from below and above, i.e.
$0<\inf_x f^-(x)<\sup_x f^+(x)<\infty$, (for example, if $\X$ is
finite) then the constant $C_s$ in (\ref{bound1}) is non-random and
independent of $s$. We see that 1. relaxes the first requirement of
(\ref{strHMM2}), because (under irreducibility) now all elements of
$\mathbb{P}^R$ must be non-negative. For aperiodic chain, such $R$
always exists and so 1. is not restrictive. On the other hand, the
assumption on emission densities is stronger, because they all must
be strictly positive. When densities are all positive, then there is
only one cluster $C=\Y$, hence under 1. and 2. above, the
cluster-assumption holds. The assumption 2. about the positivity of
densities is often made in literature (e.g..
\cite{ChigSPA,forgInit1,forgInit3}). In particular, it is the
HMM-version of the nondegeneracy-assumption  \cite{vH09b,vH15}. Of
course, the above-mentioned articles deal with continuous state
space $\X$, where the technique is different. However, at least in
finite state case, the mutual equivalence of emission distributions
excludes many important models and can be restrictive. The cluster
assumption, however, combines the zero-densities and
zero-transitions, being therefore applicable for much larger class
of models.
\\\\
Another assumption of similar type, originally also applied in the case of finite ${\cal Y}$, can be found in
\cite{LeGlandMevel,Ung2}: the matrix $\mathbb{P}$ is primitive and
\begin{equation*}\label{LeG}
\int_{\X}{\min_i f_i(x)\over \max_i f_i(x)}f_j(x)\mu(dx)>0,\quad
\forall j\in S.\end{equation*} This assumption relaxes the
requirement of positive densities, but it  implies that $\mu\{x:
\min_i f_i(x)>0\}>0$ so that $\Y$ is a
cluster that satisfies cluster assumption.\\\\
Although we have seen that  the cluster assumption is weaker than
many mixing assumptions in the literature, it is still strictly
stronger than {\bf A1} and {\bf A2}. To illustrate this fact,
consider a following example (a modification of Example 5.1 in
\cite{IEEE}) of four state HMM with transition matrix
$$\left(
                                  \begin{array}{cccc}
                                    1/2 & 0 & 1/4 & 1/4 \\
                                    0 & 1/2 & 1/4 & 1/4 \\
                                    1/2 & 0 & 1/2 & 0 \\
                                    0 & 1/2 & 0 & 1/2\\
                                  \end{array}
                                \right).$$
Suppose $G_1=G_2$, $G_3=G_4$, $G_1\cap G_3=\emptyset$. There are two
clusters: $C_1=\{1,2\}$ and $C_2=\{3,4\}$, the corresponding
sub-transition matrices are not primitive. Thus cluster-assumption
fails. To see that {\bf A1} and {\bf A2} hold, take $\X=\{1,2\}$ and
$f_1(1)=f_2(1)=1, f_3(2)=f_4(2)=1$. Now, take $x_{1:3}=112$ and
observe that $p_{13}(x_{1:3})=p_{11}f_1(1)p_{13}f_3(2)={1\over 8}$.
Similarly, it holds that
$$p_{14}(x_{1:3})=p_{23}(x_{1:3})=p_{24}(x_{1:3})=p_{33}(x_{1:3})=p_{34}(x_{1:3})=p_{43}(x_{1:3})=p_{44}(x_{1:3})=1/8.$$ Since
$f_1(2)=f_2(2)=0$, we have
$$p_{11}(x_{1:3})=p_{12}(x_{1:3})=p_{21}(x_{1:3})=p_{22}(x_{1:3})=p_{41}(x_{1:3})=p_{42}(x_{1:3})=p_{31}(x_{1:3})=p_{32}(x_{1:3})=0.$$
Thus $\Y^+(x_{1:3})=\Y\times \{3,4\}$ and hence {\bf A1} and {\bf
A2} hold.\\\\
We conclude the section with some examples of assumptions made in
the literature that are weaker than cluster assumption (or not
comparable with it), but still stronger than {\bf A1} and {\bf A2}.
First of them is Assumption 4.3.31 in \cite{HMMbook}.  When adapted
to our case of discrete ${\cal Y}$, one of the main conditions of this assumption is (there
are also some other conditions, making it more stronger)  as
follows: there exists a $\mu$-a.s. non-identically null function
$\alpha \colon \X \rightarrow [0,1]$ and $C \subset \Y$ such that
for all $i,j\in \Y$ and for all $x \in \X$
\begin{align*}
\dfrac{\sum_{k \in C} p_{ik} f_k(x) p_{kj}}{\sum_{k \in
\Y}p_{ik}f_k(x) p_{kj}} \geq \alpha(x).
\end{align*}
This condition implies \textbf{A1}. Indeed, let $C' \subset \Y$ be a
cluster.  Then there exists $\X'$ such that $\mu(\X')>0$ and
$f_i(x)>0$ for $x \in \X'$ if and only if $i \in C'$. Take $E=\X
\times \{x \:| \: \alpha(x)>0\} \times \X'$. Thus for $x_{1:3} \in
E$,
\begin{align*}
\Y^+(x_{1:3})=\left\{(i,j) \: \middle| \:
\max_{k}p_{ik}f_k(x_2)p_{kj}f_j(x_3)>0 \right\}=\Y \times C',
\end{align*}
and so \textbf{A1} holds. It is also implicitly assumed that $\{x \:
| \: \alpha(x)>0 \}$ is $\mu$-positive, whence \textbf{A2} also
follows. This assumption is not comparable with cluster assumption.
\\\\
Another example of the kind can be found in \cite{DMR}, where one of
the main conditions, when adapted to our case of discrete ${\cal Y}$ (the article \cite{DMR} deals with state-space models), is the following:
there exists a state $l$ such that $p_{il}(x_{1:r-1})p_{lj}>0$ for
every $i,j\in \Y$.  The state $l$ is called {\it uniformly
accessible}. Clearly this condition is of type (\ref{sopot}) and as
argued in the Remark in Subsection \ref{sec:dobrushin}, slightly
stronger that ${\bf A1}$. Interestingly, although the methods in
\cite{DMR} are different as the ones in our paper (coupling), the
same kind of condition appears.
\\\\
Yet another way to re-define the cluster assumption is the
following: let $C\subset \Y$ be a cluster, but the matrix
$\mathbb{P}_C$ satisfies the following assumption: every column of
$\mathbb{P}_C$ either consists of strictly positive entries or has
all entries equal to 0. Such matrix satisfies Doeblin condition and
therefore the set $C$ is sometimes called to have {\it local
Doeblin} property. In \cite{forgInit1,forgInit3}, this condition is applied for general state-space ${\cal Y}$, our statement is again the discrete ${\cal Y}$ version of it. If all entries are
positive, then $\mathbb{P}_C$ is primitive (and cluster condition
holds), otherwise not. To see that {\bf A1} and {\bf A2} still hold,
construct the set $E\in \X^3$ as in the proof of Lemma
\ref{lem:HMMcond} with $R=1$. Then $\Y^+=\Y_C \times C'$, where
$C'\subset C$ is the set of states corresponding to non-zero
columns.  This kind of assumption appears in \cite{forgInit1}. In
\cite{forgInit3} it is strengthen so that to every observation $x$
corresponds a local Doeblin set that satisfies (\ref{weak-cl}).
\\\\
An interesting and easily verifiable sufficient condition for the
filter stability (\ref{tilde1}) is proven in \cite{ChigFilter}: the
transition matrix has to be primitive with at least one row
consisting of  all non-zero entries \cite[Ex. 1.1]{ChigFilter}. This assumption is not
comparable with cluster assumption, because the latter can be
fulfilled with a matrix having zero in every row, and vice versa. On
the other hand, it does not assume anything about the emission
densities and so it is very practical. We shall show that ${\bf
A1}-{\bf A2}$ still hold.
\begin{proposition} If $\mathbb{P}$ is irreducible and has a least
one row consisting of non-zero entries, then {\bf A1} and {\bf A2}
hold.\end{proposition}
\begin{proof} Let the first row of  $\mathbb{P}$ consisting of
strictly positive entries. Since every state belongs to at least one
cluster, let $C_1$ be a cluster containing 1. In what follows, for a
cluster $C$, let $G_C=(\cap _{i\in C}G_i)\setminus (\cup_{j\not\in
C} G_j)$. We construct the set $E$ as follows. Take $E_1=\X\times
F_1$, where $F_1=G_{C_1}$ and notice that $\Y^+(x_{1:2})$ is the
same for every $x_{1:2}\in E_1$. Indeed, if $(i,j)\in
\Y^+(x_{1:2})$, then  $p_{ij}f_j(x_2)>0$,
and if this holds for a $x_2\in F_1$, then it holds for any other
$x'_2\in F_1$ as well. Observe that due to the assumption $1\in
\Y^+_{(2)}$. Relabel the states so that
$\Y^+_{(2)}=\{1,2,\ldots,l\}$. Let $A_1\subset \Y^+_{(1)}$ be the
set of states that can be connected with 1. Formally, $i\in A_1$ if
$p_{i1}(x_{1:2})>0$ for every $x_{1:2}\in E_1$. Clearly $A_1\ne
\emptyset$. If $A_1=\Y^+_{(1)}$, then the proposition is proved --
just take $E=E_1\times F_1$ and observe that by assumption for any state $k$ in $C_1$, $p_{1k}>0$. Let $A_2=\Y^+_{(1)}\setminus A_1$
consists of states that cannot be connected to 1 but can be
connected to 2. Thus $i\in A_2$, if and only if $p_{i2}(x_{1:2})>0$,
but $p_{i1}(x_{1:2})=0$ for every $x_{1:2}\in E_1$. The set $A_2$
might be empty. Similarly define
$$A_k\DEF \{i\in \Y^+_{(1)}\setminus (\cup_{j=1}^{k-1} A_i):
p_{ik}(x_{1:2})>0,\quad \forall x_{1:2}\in E_1\},\quad
k=3,\ldots,l.$$ By irreducibility there exists a path
$i_1,i_2,\ldots,i_s$, with $i_1=2$ and $i_s=1$ from the state 2 to
the state 1. Let $C_2,\ldots C_s$ be the corresponding clusters
containing $i_2,\ldots, i_s$ and define $F_j=G_{C_j}$,
$j=2\ldots,s$. Finally take $E_2=F_2\times \cdots \times F_s$. Since
$p_{1 i_2}>0$ by assumption (the first row has all non-zero entries), we have that for every $x_{2:s}\in
E_2$, $p_{1 1}(x_{2:s})>0$ and $p_{2 1}(x_{2:s})>0$. Now enlarge the
set $E_1$ by taking $E_1\times E_2$ and redefine the sets
$A'_1,A'_2,\ldots A'_{l'}$. Observe: if for a $k=1,\ldots,l $ and
for $x_{2:s}\in E_2$, $p_{k,1}(x_{2:s})>0$ then $A_k\subset A'_1$.
Therefore $A_1\cup A_2\subset A'_1$, and $l'<l$.  If $l'>1$, then
proceed similarly by enlarging $E_1\times E_2$ until all elements of
$\Y^+_{(1)}$ can be connected with 1. This proves {\bf A1}. The
assumption {\bf A2} is trivial.\end{proof}

\subsection{Linear Markov switching model}
Let $\X= \mathbb{R}^d$ for some $d \geq 1$ and for each state $i \in
\Y$ let $\{\xi_k(i)\}_{k \geq 2}$ be an i.i.d. sequence of random
variables on $\X$ with $\xi_2(i)$ having density $h_i$ with respect
to Lebesgue measure on $\mathbb{R}^d$. We consider the {\it linear
Markov switching model}, where $X$ is defined recursively by
\begin{align} \label{LMSM}
X_{k}=F(Y_k)X_{k-1}+ \xi_k(Y_k), \quad k \geq 2.
\end{align}
Here $F(i)$ are some $d \times d$ matrices, $Y=\{Y_k\}_{k\geq 1}$ is
a Markov chain with transition matrix $(p_{ij})$, $X_1$ is some
random variable on $\X$, and random variables $\{\xi_k(i)\}_{k \geq
2, \: i \in \Y}$ are assumed to be independent and independent of
$X_1$ and $Y$. For the linear switching model measure $\mu$ is
Lebesgue measure on $\mathbb{R}^d$ and the transition kernel density
expresses as $q(x_2,j|x_1,i)=p_{ij}h_j(x_2-F(j)x_1)$. When $F(i)$
are zero-matrices, then the linear Markov switching model simply
becomes HMM with $h_i$ being the emission densities. When $F(i)=F$
for every $i\in \Y$, then the model becomes autoregressive model
with correlated noise. Linear Markov switching models, also
sometimes called {\it linear autoregressive switching models} have
been widely used in econometric modelling, see e.g.
\cite{MSM1,MSM2,MSM3}.
\\\\
The following result gives sufficient conditions for
\textbf{A1}-\textbf{A2} to hold. The analytic form of the stationary
density $p(z_1)$ is usually intractable for the linear switching
model, and therefore we will avoid its use in the conditions.
Instead, we will rely solely on the notion of $\psi$-irreducibility.
In what follows, let $\| \cdot \|$ denote the 2-norm on
$\X=\mathbb{R}^d$, and for any $x \in \X$ and $\epsilon>0$ let
$B(x,\epsilon)$ denote an open ball in $\X$ with respect to 2-norm
with center point $x$ and radius $\epsilon>0$.
\begin{lemma} \label{lem:switchCond}Let $Z$ be a $\psi$-irreducible linear Markov switching model. If the following conditions are fulfilled, then $Z$ satisfies \textbf{A1}-\textbf{A2}.
\begin{enumerate}[label=(\roman*)]
\item  \label{LMprimit}There exists set $C\subset \Y$ and $\epsilon>0$ such that the following two conditions are satisfied:
\begin{enumerate}
\item[1.] for $x \in B(0,\epsilon)$, $h_i(x)>0$ if and only if $i \in C$;
  \item[2.] the sub-stochastic matrix $(p_{ij})_{i,j\in C}$ is primitive.
\end{enumerate}
\item \label{LMreach} Denote $\Y_C=\{i \in \Y \:| \exists j\in C: \: p_{ij}>0\}$. There exists $i_0 \in \Y_C$ such that $(0,i_0) \in {\rm supp}(\psi)$.
\end{enumerate}
\end{lemma}
\begin{proof}There must exist $\epsilon_0 >0$ such that
\begin{align}\label{lessr}
\|x-F(j)x'\|<\epsilon, \quad \forall j \in \Y, \quad \forall x,x'
\in B(0,\epsilon_0).
\end{align}
By \ref{LMprimit} there exists $R \geq 1$ such that $\mathbb{P}_C^R$
contains only positive elements. We take $E=B(0,\epsilon_0)^{R+2}$.
Fixing $x_{1:R+2} \in E$, we have for any $i,j \in \Y$
\begin{align*}
p_{ij}(x_{1:R+2})&=\sum_{y_{1:R+2} \colon
(y_1,y_{R+2})=(i,j)}\prod_{k=2}^{R+2}p_{y_{k-1}y_k}h_{y_k}(x_k-F(y_k)x_{k-1}).
\end{align*}
Together with \eqref{lessr} and \ref{LMprimit} this implies that
$p_{ij}(x_{1:R+2})>0$ if and only if $i \in \Y_C$ and $j \in C$.
Hence $\Y^+(x)=\Y_C \times C$ for every $x \in E$. Together with
\ref{LMreach} this implies that \textbf{A1}-\textbf{A2} hold.
\end{proof}
\\\\
Note that  if densities $h_i$ are all positive around 0 (for
example, Gaussian), then \ref{LMprimit} is fulfilled when
$\mathbb{P}$ is primitive with $C=\Y$. General conditions for the
linear Markov switching model to be positive Harris and aperiodic
can be found in \cite{linearErgodic}.
\\\\
{\bf Remark:} Instead of the linear Markov switching model, we can
also consider the general Markov switching model, also called the
{\it nonlinear autoregressive switching model}. For this model the
linear recursion in \eqref{LMSM} is replaced by any measurable
function $G \colon \Y \times \X \rightarrow \X$, i.e.
\begin{align*}
X_{k}=G(Y_k, X_{k-1})+ \xi_k(Y_k), \quad k \geq 2,
\end{align*}
The statement of Lemma \ref{lem:switchCond} holds for this model as
well, if we demand that the $G(i,\cdot)$ satisfy the following
additional conditions:
\begin{align} \label{Gcond}
\mbox{$G(i,\cdot)$ are continuous at 0, and $G(i,0)=0$ for all $i
\in \Y$.}
\end{align}
If these conditions are too restrictive, a different approach is
needed to prove \textbf{A1}-\textbf{A2}. For general conditions for
positivity, Harris recurrence and aperiodicity of the non-linear
switching model see e.g. \cite{linearErgodic, nonlinSwitch1,
nonlinSwitch2}.

\begin{appendices}
\section{}
\begin{lemma} \label{lem:io} Let $Z$ be Harris recurrent. If some measurable set $W \subset \Z^n$, $n \geq 2$, satisfies  $\psi(W_{(1)})>0$ and
\begin{align*}
 P_z(Z_{1:n} \in W)>0, \quad \forall z \in W_{(1)},
\end{align*}
then  for all $z \in \Z$
\begin{align*}
P_z(Z \in W \mbox{ i.o.}) \DEF P_z \left(\bigcap_{k=1}^\infty
\bigcup_{l=k}^\infty \{Z_{l:l+n-1} \in W \} \right)=1.
\end{align*}
\end{lemma}
\begin{proof} By the same argument as in the proof of (\ref{ineq:n0Bound}), there exists a set $W' \subset W$ and $\epsilon>0$ such that $\psi(W'_{(1)})>0$ and
\begin{align*}
 P_z(Z_{1:n} \in W')\geq \epsilon, \quad \forall z \in W'_{(1)}.
\end{align*}
By \cite[Lemma A.1]{PMMinf} $P_z(Z \in W' \mbox{ i.o.})=1$ for all
$z \in \Z$, which implies the statement.
\end{proof}
\\\\
{\bf Proof of Proposition \ref{block}.} Clearly, if $Z$ is a
stationary process, then the process $\ZZ$ is stationary as well, so
that the distribution of $\ZZ_1$ (under $\pi$) is invariant
probability measure for $\ZZ$.
\\\\
We are going to show that measure $\psi_r$ is a maximal
irreducibility measure for $\ZZ$. To see that $\psi_r$ is an
irreducibility measure, suppose $A$ satisfies $\psi_r(A)>0$. There
must exist $A' \subset A$ such that $\psi(A'_{(1)})>0$ and
$P(Z_{2:r} \in A'\mid Z_1=z_1)>0$ for all $z_1 \in A'_{(1)}$. Since
$\psi(A'_{(1)})>0$ then for every $z \in \Z$ there exists $k=k(z)
\geq r+1$ such that $P(Z_{k(z)} \in A'_{(1)}|Z_r=z)>0$. Thus
$P(\ZZ_{k(z)} \in A'|Z_r=z)>0$ for every $z \in \Z$, which implies
that $P(\ZZ_{k(z)} \in A|\ZZ_1=(z_{1:r-1},z))>0$ for every
$(z_{1:r-1},z) \in \Z^r$, and so $\pp$ is an irreducibility measure.
\\\\
 To show that $\pp$ is a maximal irreducibility measure, we need
show that $\pp \succ {\varphi}_r$ for arbitrary irreducibility
measure ${\varphi}_r$. Suppose ${\varphi}_r(A)>0$. Then
$P_{\pi}(\ZZ_1\in A)>0$, because invariant measure dominates any
irreducibility measure \cite[Prop. 10.1.2(ii)]{MT}. Then there
exists $A'\subset A$ so that $\pi(A'_{(1)})>0$ and $$
P\big(Z_{2:r}\in A'(z_1)\mid Z_1=z_1\big)>0\quad \text{for any}\quad
z_1\in A'_{(1)}.$$ Recall that $\psi$ and $\pi$ are maximal
irreducibility and invariant measures of (Harris) recurrent chain
$Z$. Then these measures are equivalent \cite[Th.10.4.9]{MT}. Hence
$\psi(A'_{(1)})>0$ and so by definition (\ref{moot}) $\psi_r(A)\geq
\psi_r(A')>0$. Thus $\pp \succ {\varphi}$.
\\\\
It remains to show that $\ZZ$ is Harris chain. Let $A$ be such that
$\psi_r(A)>0$. By (\ref{moot}), there exists $A'\subset A$ so that
$\psi(A'_{(1)})>0$ and $P_{z_1}\big((Z_2,\ldots,Z_r)\in
A'(z_1)\big)>0$ for every $z_1 \in A'_{(1)}$. Thus with
$$B=\cup_{z_1\in A'_{(1)}} \{z_1\}\times A'(z_1)\subset A,$$ we
have $$P(Z_{1:r}\in B\mid Z_1=z_1)>0,\quad \forall z_1\in A_{(1)}.$$
Since $\psi(A_{(1)})>0$, and $Z$ is Harris, by Lemma \ref{lem:io},
it follows
 that $P(Z_k\in B, \rm{i.o})=1$.
 Thus $\ZZ$ is a Harris chain.

\end{appendices}
\paragraph{Acknowledgment.} The research is supported by Estonian  institutional research funding
IUT34-5 and PRG 865.

\bibliographystyle{plain}
\bibliography{smoothing}
\end{document}